\documentclass[oneside,11pt,a4paper,reqno,ps]{article}
\usepackage[a4paper,top=2.5cm,bottom=2.5cm,left=2.5cm,right=2.5cm,
bindingoffset=5mm]{geometry}
\usepackage[latin1]{inputenc}
\usepackage[english]{babel}
\usepackage[T1]{fontenc}
\usepackage{syntonly}
\usepackage{graphicx}
\usepackage{amssymb}	
\usepackage{lmodern}
\usepackage{amsthm}
\usepackage{microtype}
\usepackage{amsmath}
\usepackage{braket}
\usepackage{eurosym}
\usepackage{amssymb}
\usepackage{amstext}
\usepackage{color}
\usepackage{xcolor}
\usepackage{url}
\usepackage{amsmath,amsthm,amssymb,amscd,epsfig,esint, mathrsfs,graphics,color}
\usepackage{wrapfig}
\usepackage{verbatim}
\usepackage{hyperref}
\usepackage{listings}
\usepackage{bookmark}
\usepackage{csquotes}
\hypersetup{%
	pdfpagemode={UseOutlines},
	bookmarksopen,
	pdfstartview={FitH},
}
\usepackage[backend=bibtex, hyperref]{biblatex}
\newcommand{\numberset}{\mathbb}
\newcommand{\R}{\numberset{R}}

\newcommand{\lAbs}{\left|}
\newcommand{\rAbs}{\right|}

\theoremstyle{plain}
\newtheorem{thm}{Theorem}[section]

\newtheorem{proposition}[thm]{Proposition}
\newtheorem{lemma}[thm]{Lemma}
\newtheorem{corollary}[thm]{Corollary}

\theoremstyle{definition}
\newtheorem{definition}[thm]{Definition}

\newtheorem{remark}[thm]{Remark}

\def\XXint#1#2#3{{\setbox0=\hbox{$#1{#2#3}{\int}$} 
		\vcenter{\hbox{$#2#3$}}\kern-.5\wd0}}

\def\div{{\rm div}}

\def\supp{\mathrm{supp}}
\def\loc{\mathrm{loc}}
\numberwithin{equation}{section} \makeatletter

\renewcommand{\p@enumi}{\thesection.}
\title{\textbf{Higher differentiability results for solutions to a class of non-homogeneouns elliptic problems under sub-quadratic growth conditions}}
\author{Albert Clop, Andrea Gentile, Antonia Passarelli di Napoli} 

\begin{document}
	\maketitle
	\begin{abstract}
		\noindent{We prove  a sharp higher differentiability result for local minimizers of functionals of the form
			$$\mathcal{F}\left(w,\Omega\right)=\int_{\Omega}\left[ F\left(x,Dw(x)\right)-f(x)\cdot
			w(x)\right]dx$$
			with {non autonomous} integrand $F(x,\xi)$ which is convex with respect to the gradient variable, under $p$-growth conditions, with $1<p<2$. The main novelty here is that the results are obtained assuming that the partial map $x\mapsto D_\xi F(x,\xi)$ has weak derivatives in some Lebesgue space $L^q$ and the datum $f$ is assumed to belong to  a suitable Lebesgue space $L^r$.\\
		We also prove that it is possible to weaken the assumption on the datum $f$ and on the map $x\mapsto D_\xi F(x,\xi)$, if the minimizers are assumed to be a priori bounded.}
	\end{abstract}
	
	\bigskip
	
	{\footnotesize{ \emph{Mathematics Subject
				Classification}. {35J47; 35J70; 49N60.}
			
			}}
	
	\bigskip
	\section{Introduction}
	Let $\Omega\subset\R^n$ be a bounded open set, with $n>2$. For $N\geq 1$ and $w:\Omega\to \R^N$, we consider the following non autonomous, non homogeneous functional 
	\begin{equation}\label{modenergy}
		\mathcal{F}\left(w,\Omega\right)=\int_{\Omega}\left[ F\left(x,Dw(x)\right)-f(x)\cdot
		w(x)\right]dx
	\end{equation}
In this paper we deal with the regularity properties of local minimizers of $\mathcal{F}$. To fix the ideas, we assume that $f\in L^r_\loc(\Omega, \R^N )$ is given, and $2<r<n$. The Carath\'eodory function $F: \Omega\times \R^{n\times
	N}\to [0,+\infty)$ is such that $\xi\mapsto F(x,\xi) $ is $C^2\left(\R^{n\times N}\right)$ for a.e. $x\in\Omega$. Morever, we assume that there exists real numbers $p\in(1, 2)$ and $\mu\in[0, 1]$ such that the following set of assumptions is satisfied:
\begin{itemize}
\item There exist positive constants $\ell, L$ such that
\begin{equation}\label{F1}
	\ell \left(\mu^2+\left|\xi\right|^2\right)^\frac{p}{2}\le F(x,\xi)\le L\left(\mu^2+\left|\xi\right|^2\right)^\frac{p}{2}
\end{equation}
for a. e. $x\in\Omega$ and every $\xi\in\R^{n\times N}$.
\item There exists a positive constant $\nu>0,$ such that
\begin{equation}\label{F2}
	\langle D_{\xi}F(x,\xi)-D_{\xi}F(x,\eta),\xi-\eta\rangle \ge \nu\left(\mu^2+\left|\xi\right|^2+\left|\eta\right|^2\right)^\frac{p-2}{2}|\xi-\eta|^{2}
\end{equation}
for a. e. $x\in\Omega$ and every $\xi, \eta\in \R^{n\times N}$.
\item There exists a positive constant $L_1>0,$ such that
\begin{equation}\label{F3}
	\left|D_{\xi}F(x,\xi)-D_{\xi}F(x,\eta)\right| \le L_1\left(\mu^2+\left|\xi\right|^2+\left|\eta\right|^2\right)^\frac{p-2}{2}|\xi-\eta|
\end{equation}
for a. e. $x\in\Omega$ and every $\xi, \eta\in \R^{n\times N}$.
\item There exists a non-negative function $g\in L^{q}_{\loc}\left(\Omega\right)$ such that
\begin{equation}\label{F4}
	\left|D_{\xi}F\left(x,\xi\right)-D_{\xi}F\left(y,\xi\right)\right|\le\left(g(x)+g(y)\right)\left(\mu^2+\left|\xi\right|^2\right)^\frac{p-1}{2}\left|x-y\right|
\end{equation}
for a.e. $x, y\in\Omega$ and  every $\xi\in\R^{n\times N}$.
\end{itemize}
In the last years, there has been great interest in understanding the differentiability of minimizers under some weak differentiability assumption for $F$ in the $x$ variable. Examples of this are the works by Passarelli di Napoli \cite{32,33} for the case $p\ge2$. In this paper, we intend to explain how the situation changes if $p$ stays in the lowest range of growths, that is, $1<p<2$. \\
\\
Among the big quantity of previous contributions, we wish to mention works by Di Benedetto \cite{DiB} and Manfredi \cite{M}, all in mid 80's, in which it is proven that local minimizers have H\"older continuous derivatives under different instances, all of them involving autonomous functionals $F$ (i.e. non depending on $x$) with growth exponent $p>1$, and with independent term $f=0$.  Next, in \cite{gm}, the H\"{o}lder continuity of the gradient of solutions has been proved in case $p\ge2$ and H\"{o}lder continuous coefficients with $f=0$.\\ 
It was not until the work of Acerbi and Fusco \cite{AF} that an extension of this result was provided for functionals $F$ that could depend on $x$ in a H\"older continuous manner. In order to prove the Lipschitz continuity of local minimizers, in the same paper (\cite{AF}), for $1<p<2$ but still for $f=0$, the second order regularity  of the local minimizers is established in case of constant coefficients. Very recently, A. Gentile \cite{Gentile2} extended the higher differentiability result of \cite{AF} to the case when the dependence of $F$ in $x$ is of Sobolev type. In presence of a bounded independent term $f$,  Tolksdorf studied functionals $F$ that are Lipschitz continuous in the space variable, and obtained that local minimizers are H\"older continuous.\\
In case of degenerate elliptic functionals, we mention the recent papers \cite{BRASCO2010652, COLOMBO201494, ColomboFigalli1}, in wich Lipschitz regularity of solutions is established assuming $f\in L^r$, with $r>n$.
\\
In the present paper, our first goal  is to get for the minimizers not the H\"older continuous first derivatives, but up to $L^p$ second order derivatives. We wish to do this in two ways. First, by requiring $F$ the minimal possible regularity in $x$. Second, by finding optimal conditions on $f$ that make the whole scheme work.\\
\\
Actually, the first result we prove in this paper is the following. Below, we denote $V_p(\xi)=\left(\mu^2+|\xi|^2\right)^\frac{p-2}{4}\xi$.

\begin{thm}\label{CGPThm1}
Let $\Omega\subset\R^n$ be a bounded open set, and $1<p<2$. Let $u\in W^{1, p}_{\loc}(\Omega, \R^N)$ be a local minimizer of the functional \eqref{modenergy}, under the assumptions \eqref{F1}--\eqref{F4}, with  

$$f\in L^{\frac{np}{n(p-1)+2-p}}_{\loc}\left(\Omega\right)\qquad\mbox{ and }\qquad g\in L^{n}_{\loc}\left(\Omega\right).$$

Then $V_p\left(Du\right)\in W^{1, 2}_{\loc}\left(\Omega\right)$, and the estimate
\begin{equation}\label{mainestimateVp}
\aligned
&\int_{B_{\frac{R}{2}}}\left|D\left(V_p\left(Du(x)\right)\right)\right|^2dx \\
&\leq\frac{c}{R^{\beta\left(n, p\right)}}\left[\int_{B_{R}} \left(\mu^2+\left|Du(x)\right|^2\right)^\frac{p}{2} +\int_{B_R} |f(x)|^\frac{np}{n\left(p-1\right)+2-p}dx +c\int_{B_{R}}g^{n}(x)dx +\left|B_R\right|\right],
\endaligned
\end{equation}
holds true for any ball $B_{R}\Subset\Omega$, with $\beta(n, p)>0.$
\end{thm}

At this point, it is worth mentioning that $|D^2u|\simeq\left|D\left(V_p\left(Du\right)\right)\right|\left(\mu^2+\left|Du\right|^2\right)^\frac{2-p}{4}$ so that the above theorem establishes a weighted estimate for the second derivative. Furthermore, from Young's inequality one has that $\left|D^2u\right|^p\simeq\left|D\left(V_p\left(Du\right)\right)\right|^2+\left(\mu^2+\left|Du\right|^2\right)^\frac{p}{2}$ so that, in particular, $V_p\left(Du\right)\in W^{1,2}$ implies $u\in W^{2,p}$. See Section \ref{auxiliaryfunction} for details.

Also, for what concerns the assumptions on $f$, we observe that
$$
2<\frac{np}{n\left(p-1\right)+2-p}<n,
$$
for any $n>2$ and $1<p<2$. The proof of Theorem \ref{CGPThm1} is based on the combination of an a priori estimate and a suitable approximation argument. In order to achieve the a priori estimate under the sharp integrability assumption on the independent term $f$, besides the use of the difference quotient method, we need to apply carefully the well known iteration Lemma on concentric balls to control the terms with critical integrability.  We'd like to mention that previous higher differentiability result in the subquadratic non standard growth case has been recently obtained in \cite{MPdN} with independent term $f\in L^{\frac{p}{p-1}}$. Especially, note that
$$\frac{np}{n(p-1)+2-p}<\frac{p}{p-1}\iff 1<p<2.$$
Also in case of degenerate elliptic equations with sub-quadratic growth, in the recent paper \cite{AMBROSIO2022125636}, a fractional higher differentiability has been obtained under a Besov or Sobolev assumption on the datum $f$.\\
In \cite{CGHP}, an higher differentiability result has been established for local minimizers of \eqref{modenergy} in the case $p\ge 2$ and for $f\in L^n\log^\alpha L$, with $\alpha>0$. For other Lipschitz regularity and higher differentiability results for solutions to non-homogeeneous elliptic problems we also refer to \cite{Beck-Mingione, de2021lipschitz} and to \cite{kuusi2018vectorial} for regularity results of solutions to problems with measure data.
\\
We want to stress that, for what concerns the regularity of the funtion $f$, Theorem \ref{CGPThm1} is a sharp result: it is not possible to weaken the assumption on the integrability of the datum, as we will see with a counterexample in Section \ref{Counterexample} below.\\
\\
Independently to the previous problem, a new interest has arosen in the last years. It consists in describing the regularity properties of local minimizers which one assumes apriori bounded. The reason for this is that many times local boundedness of minimizers is available much before any sort of weak differentiability. Also, as it will be clear from our results, apriori boundedness of minimizers helps in relaxing the assumptions for $f$, at least when $n$ is not too small. Results in this direction are available so far just for $p\ge2$. An example of this can be found in Carozza, Kristensen and Passarelli di Napoli \cite{CKP}, where $F$ is assumed Lipschitz-continuous in $x$ (see also \cite{capone2020regularity}). Similar analysis have been done by Giova and Passarelli di Napoli in \cite{GiovaPassarelli} assuming only Sobolev regularity for $F$ in the $x$ variable. The second goal in the present paper consists of exploring if this kind of results is available also under the assumption $1<p<2$. We obtained the following.

\begin{thm}\label{inftythm}
	Let $\Omega\subset\R^n$ be a bounded open set, $1<p<2$ and $u\in W^{1, p}_{\loc}\left(\Omega, \R^N\right)\cap L^\infty_{\loc}\left(\Omega\right)$ be a local minimizer of the functional \eqref{modenergy} under assumptions \eqref{F1}--\eqref{F4} with
	
	$$f\in L^{\frac{p+2}{p}}_{\loc}\left(\Omega\right)\qquad\mbox{ and }\qquad g\in L^{p+2}_{\loc}\left(\Omega\right).$$
	
	Then $V_p\left(Du\right)\in W^{1,2}_{\loc}\left(\Omega\right)$ and the estimate
\begin{equation}\label{inftyestimate}
\aligned
&\int_{B_{\frac{R}{2}}}\left|D\left(V_p \left(Du(x)\right) \right)\right|^2dx\\
&\leq\frac{c\|u\|_{L^\infty(B_{4R})}}{R^{\frac{p+2}{p}}}
\left[
\int_{B_{4R}} \left(\mu^2+\left|Du(x)\right|^2\right)^\frac{p}{2}dx+\int_{B_{R}}g^{p+2}(x)dx+\int_{B_R}\left|f(x)\right|^\frac{p+2}{p}dx+\left|B_R\right|+1
\right]
	\endaligned
	\end{equation}
	holds for any ball $B_{4R}\Subset\Omega$.
\end{thm}
Notice here that for $n>2$ and $1<p<2$ we have $2<\frac{p+2}{p}<n$ . Also, by comparing the assumptions on $f$ both in Theorems \ref{CGPThm1} and \ref{inftythm}, it is worth noticing that for $1<p<2$ one has
$$
\frac{p+2}{p}<\frac{np}{n(p-1)+2-p}\hspace{1cm}\Longleftrightarrow\hspace{1cm}n>p+2.
$$  
Therefore, Theorem \ref{inftythm} improves Theorem \ref{CGPThm1} whenever $n\ge4$.\\ 
In proving Theorem \ref{inftythm}, once  the a priori estimate is established, the more delicate issue is to construct some approximating problems in a convenient way. For this, these aproximating problems need to be smooth with respect to the dependence on the $x$-variable and on the datum $f$. Also, they need to have minimizers whose $L^{r}$ norm is close to the $L^\infty$ norm of the minimizer of the original problem for $r$ sufficiently large. We overcome this difficulties by using the penalization method introduced  in \cite{CKP}. Still, we need to prove second order estimates for the approximating minimizers which, as far as we know, are available only for $p\ge 2$ ( see \cite{GiovaPassarelli}).
\\
The paper is organized as follows. In section 2 we collect definitions and preliminary results.  Section 3,contains  the proof of Theorem \ref{CGPThm1}. In Section 4, we give the counterexample showing the optimality of  the assimption on the datum. In Section 5 we give the proof of the higher differentiability of the local minimizers of a class of variational integrals with a singular penalisation term. In Section 6, we give the proof of Theorem \ref{inftythm}.

\medskip

\section{Preliminary results}\label{preliminaryresults}
	
	We will  follow the usual convention and denote by $c$ or $C$ a
	general constant that may vary on different occasions, even within
	the same line of estimates. Relevant dependencies on parameters and
	special constants will be suitably emphasized using parentheses or
	subscripts. All the norms we use  will be the standard Euclidean
	ones and denoted by $|\cdot |$ in all cases. In particular, for
	matrices $\xi$, $\eta \in \R^{n\times N}$ we write $\langle \xi,
	\eta \rangle : = \text{trace} (\xi^T \eta)$ for the usual inner
	product of $\xi$ and $\eta$, and $| \xi | : = \langle \xi,
	\xi\rangle^{\frac{1}{2}}$ for the corresponding euclidean norm. By
	$B_r(x)$  we will denote the ball in $\R^n$ centered at $x$
	of radius $r$. The integral mean of a function  $u$ over a ball
	$B_r(x)$  will be denoted by $u_{x,r}$, that is
	$$ u_{x,r}:=\frac{1}{\left|B_r(x)\right|}\int_{B_r(x)}u(y)dy,$$
	where $|B_r(x)|$ is the Lebesgue measure of the ball in
	$\mathbb{R}^{n}$. If no confusion  arises, we shall omit the
	dependence on the center.

	The following lemma has important applications in the so called
	hole-filling method. Its proof can be found, for example, in
	\cite[Lemma 6.1]{23} .
	\medskip
	\begin{lemma}\label{iter} Let $h:[r, R_{0}]\to \mathbb{R}$ be a nonnegative bounded function and $0<\theta<1$,
		$A, B\ge 0$ and $\gamma>0$. Assume that
		$$
		h(s)\leq \theta h(t)+\frac{A}{\left(t-s\right)^{\gamma}}+B,
		$$
		for all $r\leq s<t\leq R_{0}$. Then
		$$
		h(r)\leq \frac{c A}{(R_{0}-r)^{\gamma}}+cB ,
		$$
		where $c=c(\theta, \gamma)>0$.
	\end{lemma}

The following Gagliardo-Niremberg type inequalities are stated in
\cite{GiovaPassarelli}. For the proofs see the Appendix A of \cite{CKP} and Lemma
3.5 in \cite{GiannettiPassa} (in case $p(x)\equiv p, \, \forall x$)
respectively.

\begin{lemma}\label{lemma5GP}
	For any $\phi\in C_0^1(\Omega)$ with $\phi\ge0$, $\mu\in[0, 1]$, and any $C^2$ map $v:\Omega\to\R^N$, we have
	
	\begin{eqnarray}\label{2.1GP}
		&&\int_\Omega\phi^{\frac{m}{m+1}(p+2)}(x)\left|Dv(x)\right|^{\frac{m}{m+1}(p+2)}dx\cr\cr
		&\le&(p+2)^2\left(\int_\Omega\phi^{\frac{m}{m+1}(p+2)}(x)\left|v(x)\right|^{2m}dx\right)^\frac{1}{m+1}\cr\cr
		&&\cdot\left[\left(\int_\Omega\phi^{\frac{m}{m+1}(p+2)}(x)\left|D\phi(x)\right|^2\left(\mu^2+\left|Dv(x)\right|^2\right)^\frac{p}{2}dx\right)^\frac{m}{m+1}\right.\cr\cr
		&&\left.+n\left(\int_\Omega\phi^{\frac{m}{m+1}(p+2)}(x)\left(\mu^2+\left|Dv(x)\right|^2\right)^\frac{p-2}{2}\left|D^2v(x)\right|^2dx\right)^\frac{m}{m+1}\right],
	\end{eqnarray}
	
	for any $p\in(1, \infty)$ and $m>1$. Moreover
	
	\begin{eqnarray}\label{2.2GP}
		&&\int_{\Omega}\phi^2(x)\left(\mu^2+\left|Dv(x)\right|^2\right)^\frac{p}{2}\left|Dv(x)\right|^2dx\cr\cr
		&\le&c\left\Arrowvert v	\right\Arrowvert_{L^\infty\left(\supp(\phi)\right)}^2\int_\Omega\phi^2(x)\left(\mu^2+\left|Dv(x)\right|^2\right)^\frac{p-2}{2}\left|D^2v(x)\right|^2dx\cr\cr
		&&+c\left\Arrowvert v\right\Arrowvert_{L^\infty\left(\supp(\phi)\right)}^2\int_\Omega\left(\phi^2(x)+\left|D\phi(x)\right|^2\right)\left(\mu^2+\left|Dv(x)\right|^2\right)^\frac{p}{2}dx,
	\end{eqnarray}
	
	for a constant $c=c(p).$
\end{lemma}

By a density argument, one can easily check that estimates \eqref{2.1GP} and \eqref{2.2GP} are still true for any map $v\in W^{2,p}_{\loc}(\Omega)$ such that $\left(\mu^2+\left|Dv\right|^2\right)^\frac{p-2}{2}\left|D^2v\right|^2\in L^1_{\loc}\left(\Omega\right)$.

For further needs, we recall the following result, whose proof can be found in \cite[Lemma 4.1]{BRASCO2010652}.

\begin{lemma}\label{Lemma8}
For any $\delta>0$, $m>1$ and $\xi, \eta\in\R^k$, let
$$
W(\xi)=\left(\left|\xi\right|-\delta\right)^{2m-1}_+\frac{\xi}{\left|\xi\right|}\qquad\mbox{ and }\qquad\tilde{W}(\xi)=\left(\left|\xi\right|-\delta\right)^{m}_+\frac{\xi}{\left|\xi\right|}.
$$	

Then there exists a positive constant $c(m)$ such that

$$
\left<W(\xi)-W(\eta), \xi-\eta\right>\ge c(m)\left|\tilde{W}(\xi)-\tilde{W}(\eta)\right|^2.
$$ 

for any $\eta, \xi\in\R^k.$
\end{lemma}
\subsection{Difference quotients}\label{diffquot}
A key instrument in studying regularity properties of solutions to problems of Calculus of Variations and PDEs is the so called {\em difference quotients method}.\\
In this section, we recall the definition and some basic results.
\begin{definition}
	Given $h\in\R$, for every function
	$F:\R^{n}\to\R^N$, for any $s=1,..., n$ the finite difference operator in the direction $x_s$ is
	defined by
	$$
	\tau_{s, h}F(x)=F\left(x+he_s\right)-F(x),
	$$
	where $e_s$ is the unit vector in the direction $x_s$.
\end{definition}
In the following, in order to simplify the notations, we will omit the vector $e_s$ unless it is necessary, denoting
$$
\tau_{h}F(x)=F(x+h)-F(x), 
$$
where $h\in\R^n$.\\
\par
We now describe some properties of the operator $\tau_{h}$, whose proofs can be found, for example, in \cite{23}.

\bigskip

\begin{proposition}\label{findiffpr}
	
	Let $F$ and $G$ be two functions such that $F, G\in
	W^{1,p}(\Omega)$, with $p\geq 1$, and let us consider the set
	$$
	\Omega_{|h|}:=\Set{x\in \Omega : d\left(x,
	\partial\Omega\right)>\left|h\right|}.
	$$
	Then
	\begin{description}
		\item{$(a)$} $\tau_{h}F\in W^{1,p}\left(\Omega_{|h|}\right)$ and
		$$
		D_{i} (\tau_{h}F)=\tau_{h}(D_{i}F).
		$$
		\item{$(b)$} If at least one of the functions $F$ or $G$ has support contained
		in $\Omega_{|h|}$ then
		$$
		\int_{\Omega} F(x) \tau_{h} G(x) dx =\int_{\Omega} G(x) \tau_{-h}F(x)
		dx.
		$$
		\item{$(c)$} We have
		$$
		\tau_{h}(F G)(x)=F(x+h)\tau_{h}G(x)+G(x)\tau_{h}F(x).
		$$
	\end{description}
\end{proposition}

\noindent The next result about finite difference operator is a kind
of integral version of Lagrange Theorem.
\begin{lemma}\label{le1} If $0<\rho<R$, $|h|<\frac{R-\rho}{2}$, $1<p<+\infty$,
	and $F, DF\in L^{p}(B_{R})$ then
	$$
	\int_{B_{\rho}} |\tau_{h} F(x)|^{p}\ dx\leq c(n,p)|h|^{p}
	\int_{B_{R}} |D F(x)|^{p}\ dx .
	$$
	Moreover
	$$
	\int_{B_{\rho}} |F(x+h )|^{p}\ dx\leq  \int_{B_{R}} |F(x)|^{p}\ dx .
	$$
\end{lemma}

The following result is proved in \cite{23}.

\begin{lemma}\label{Giusti8.2}
	Let $F:\R^n\to\R^N$, $F\in L^p\left(B_R\right)$ with $1<p<+\infty$. Suppose that there exist $\rho\in(0, R)$ and $M>0$ such that
	
	$$
	\sum_{s=1}^{n}\int_{B_\rho}|\tau_{s, h}F(x)|^pdx\le M^p|h|^p
	$$
	
	for $\left|h\right|<\frac{R-\rho}{2}$. Then $F\in W^{1,p}(B_R, \R^N)$. Moreover
	
	$$
	\left\Arrowvert DF \right\Arrowvert_{L^p(B_\rho)}\le M,
	$$
	
	$$
	\left\Arrowvert F\right\Arrowvert_{L^{\frac{np}{n-p}}(B\rho)}\le c\left(M+\left\Arrowvert F\right\Arrowvert_{L^p(B_R)}\right),
	$$
	
	with $c=c(n, N, p, \rho, R)$, and 
		$$\frac{\tau_{s, h}F}{\left|h\right|}\to D_sF\qquad\mbox{ in }L^p_{\loc}\left(\Omega\right),\mbox{ as }h\to0,$$ 
		for each $s=1, ..., n.$
\end{lemma}

\subsection{An auxiliary function}\label{auxiliaryfunction}

Here we define an auxiliary function of the gradient variable that will be useful in the following.\\
The function $V_p:\R^{n\times N}\to\R^{n\times N}$, defined as 

\begin{equation*}\label{Vp}
	V_p(\xi):=\left(\mu^2+\left|\xi\right|^2\right)^\frac{p-2}{4}\xi,
\end{equation*}

\noindent for which the following estimates hold (see \cite{AF}).

\begin{lemma}\label{lemma6GP}
	Let $1<p<2$. There is a constant $c=c(n, p)>0$ such that
	
	\begin{equation}\label{lemma6GPestimate1}
		c^{-1}\left|\xi-\eta\right|\le\left|V
			_p(\xi)-V_p(\eta)\right|\cdot\left(\mu^2+\left|\xi\right|^2+\left|\eta\right|^2\right)^\frac{2-p}{4}\le c\left|\xi-\eta\right|,
	\end{equation}
\noindent for any $\xi, \eta\in\R^n.$
\end{lemma}

\begin{remark}\label{rmk1}
	One can easily check that, for a $C^2$ function $v$, there is a constant $C(p)$ such that
		\begin{equation}\label{lemma6GPestimate2}
			C^{-1}\left|D^2v\right|^2\left(\mu^2+\left|Dv\right|^2\right)^\frac{p-2}{2}\le\left|D\left(V_p\left(Dv\right)\right)\right|^2\le C\left|D^2v\right|^2\left(\mu^2+\left|Dv\right|^2\right)^\frac{p-2}{2}
		\end{equation}
	almost everywhere.
\end{remark}

In what follows, the following result can be useful.

\begin{lemma}\label{differentiabilitylemma}
	Let $\Omega\subset\R^n$ be a bounded open set, $1<p<2$, and $v\in W^{1, p}_ {\loc}\left(\Omega, \R^N\right)$. Then the implication
	\begin{equation*}\label{differentiabilityimplication}
	V_p\left(Dv\right)\in W^{1,2}_{\loc}\left(\Omega\right) \implies v\in W^{2,p}_{\loc}\left(\Omega\right) 
	\end{equation*}
	
	holds true, together with the estimate
	
	\begin{equation}\label{differentiabilityestimate}
		\int_{B_{r}}\left|D^2v(x)\right|^pdx
		\le c\cdot \left[1+\int_{B_{R}}\left|D\left(V_p\left(Dv(x)\right)\right)\right|^2+c\int_{B_R}\left|Dv(x)\right|^p\right].
	\end{equation}
	
	holds for any ball $B_R\Subset\Omega$ and $0<r<R$.
\end{lemma}

\begin{proof}
	We will prove the existence of the second-order weak derivatives of $v$ and the fact that they are in $L^p_{\loc}\left(\Omega\right)$, by means of the difference quotients method.\\
	Let us consider a ball $B_R\Subset\Omega$ and $0<\frac{R}{2}<r<R$.\\
	
	For $|h|<\frac{R-r}{2}$, we have $0<\frac{R}{2}<r<\rho_1:=r+|h|<R-|h|=:\rho_2<R$, and by \eqref{lemma6GPestimate1}, we get, for any $s=1, ..., n.$
	\begin{eqnarray*}
		\int_{B_r}\left|\tau_{s, h}Dv(x)\right|^pdx&\le& \int_{B_r}\left|\tau_{s, h}V_p\left(Dv(x)\right)\right|^p\cdot\left(\mu^2+\left|Dv(x)\right|+\left|Dv\left(x+he_s\right)\right|\right)^\frac{p\left(2-p\right)}{4}.
	\end{eqnarray*}
	
	By H\"{o}lder's Inequality with exponents $\left(\frac{2}{p}, \frac{2}{2-p}\right)$ and the use of \eqref{lemma6GPestimate1}, we get
	
	\begin{eqnarray*}
		\int_{B_r}\left|\tau_{s, h}Dv(x)\right|^pdx	&\le&\left(\int_{B_r}\left|\tau_{s, h}V_p\left(Du(x)\right)\right|^2dx\right)^\frac{p}{2}\cr\cr
		&&\cdot\left(\int_{B_{r}}\left(\mu^2+\left|Dv\left(x+he_s\right)\right|^2+\left|Dv\left(x\right)\right|^2\right)^\frac{p}{2}dx\right)^\frac{2-p}{2},
	\end{eqnarray*}
	
	and since $V_p\left(Dv\right)\in W^{1,2}_{\loc}\left(\Omega\right)$, by Lemma \ref{le1} and Young's Inequality, we have
	
	\begin{eqnarray*}\label{lemmaestimate1}
		\int_{B_r}\left|\tau_{s, h}Dv(x)\right|^pdx&\le& c\left[|h|^2\int_{B_R}\left|DV_p\left(Dv(x)\right)\right|^2dx\right]^\frac{p}{2}\cr\cr
		&&\cdot\left[\int_{B_r}\left(\mu^2+\left|Dv\left(x+he_s\right)\right|^2+\left|Dv(x)\right|^2\right)^\frac{p}{2}dx\right]^\frac{2-p}{p}\cr\cr
		&\le&c|h|^p\left[1+\int_{B_{R}}\left|DV_p\left(Dv(x)\right)\right|^2dx+\int_{B_R}\left|Dv(x)\right|^pdx\right].
	\end{eqnarray*}
	
	Since $v\in W^{1,p}_{\loc}\left(\Omega\right)$ and $V_p\left(Dv\right)\in W^{1,2}_{\loc}\left(\Omega\right)$, then, by Lemma \ref{Giusti8.2}, we get $v\in W^{2,p}_{\loc}\left(\Omega\right)$, and we have
	
	\begin{equation*}\label{differentiabilityestimate1}
		\int_{B_{r}}\left|D^2v(x)\right|^pdx
		\le c \left[1+\int_{B_{R}}\left|DV_p\left(Dv(x)\right)\right|^2dx+c\int_{B_R}\left|Dv(x)\right|^pdx\right],
	\end{equation*}
	
	that is the conclusion.
\end{proof}

	\begin{remark}\label{rmk2}
		If $\Omega\subset\R^n$ is a bounded open set and $1<p<2$, then one may use Remark \ref{rmk1} and Lemma \ref{differentiabilitylemma} to show that, if $v\in W^{1, p}_{\loc}\left(\Omega\right)$ and $V_p\left(Dv\right)\in W^{1,2}_{\loc}\left(\Omega\right)$, then $v\in W^{2, p}_{\loc}\left(\Omega\right)$ and \eqref{lemma6GPestimate2} holds true.
	\end{remark}
	\begin{remark}\label{rmk3}	
		If $\Omega\subset\R^n$ is a bounded open set and $p\in\left(1, \infty\right)$, for any  $v\in W^{1, p}_{\loc}\left(\Omega\right)$ such that $V_p\left(Dv\right)\in W^{1,2}_{\loc}\left(\Omega\right),$ if $m>1$  and $v\in L^{2m}_{\loc}\left(\Omega\right)$, then, thanks to \eqref{2.1GP}, $Dv\in L^{\frac{m\left(p+2\right)}{m+1}}_{\loc}\left(\Omega\right)$ and if $v\in L^{\infty}_{\loc}\left(\Omega\right)$, thanks to \eqref{2.2GP}, we get $Dv\in L^{p+2}_{\loc}\left(\Omega\right).$
	\end{remark}
%
%

\begin{remark}\label{rem}
For further needs we record the following elementary inequality
	\begin{equation}\label{elem}
		\left(\mu^2+\left|\xi\right|^2\right)^{\frac{p}{2}}\le 2\left(\mu^p+ \left|V_p(\xi)\right|^2\right)
	\end{equation}
	for every $\xi\in \mathbb{R}^{n\times N}$.
	\\
	Note that this is obvious if $\mu=0$.
	In case $\mu>0$, we distinguish two cases.
	If $\left|\xi\right|\le \mu$ we trivially have
	$$\left(\mu^2+\left|\xi\right|^2\right)^{\frac{p}{2}}\le 2^{\frac{p}{2}}\mu^p$$
	If $\left|\xi\right|> \mu$
		\begin{eqnarray*}
		\left(\mu^2+\left|\xi\right|^2\right)^{\frac{p}{2}}&=&\left(\mu^2+\left|\xi\right|^2\right)^{\frac{p-2}{2}}\left(\mu^2+\left|\xi\right|^2\right)\cr\cr
		&\le& \left(\mu^2+\left|\xi\right|^2\right)^{\frac{p-2}{2}}\left(\left|\xi\right|^2+\left|\xi\right|^2\right)\le 2\left(\mu^2+\left|\xi\right|^2\right)^{\frac{p-2}{2}}\left|\xi\right|^2\cr\cr
		&\le& 2\left|V_p(\xi)\right|^2
	\end{eqnarray*}
	Joining two previous inequalities we get \eqref{elem}.\\
	Moreover, if $V_p\left(Du\right)\in W^{1,2}_\loc\left(\Omega\right)$, by Sobolev's Inequality, we have $Du\in L^\frac{np}{n-2}_\loc\left(\Omega\right)=L^\frac{2^*p}{2}_\loc\left(\Omega\right)$. Indeed, using \eqref{elem}, we get
	\begin{eqnarray}\label{stinorma}
		&&\int_{B_R}\left|Du(x)\right|^{\frac{2^*p}{2}}\,dx=\int_{B_R}\left|\left|Du(x)\right|^{\frac{p}{2}-1}Du(x)\right|^{2^*}dx\cr\cr
		&\le&\mu^{\frac{2^*p}{2}}\left|B_R\right|+\int_{\Set{x\in B_R:\left|Du\right|> \mu}}\left|V_p\left(Du(x)\right)\right|^{2^*}dx\cr\cr
		&\le&\mu^{\frac{2^*p}{2}}\left|B_R\right|+\int_{B_R}\left|V_p\left(Du(x)\right)\right|^{2^*}dx,
	\end{eqnarray}
	which is finite by the Sobolev's embedding Theorem, for any ball $B_R\Subset\Omega$.
\end{remark}

\section{Proof of Theorem \ref{CGPThm1}}\label{Thm1pf}
We prove Theorem \ref{CGPThm1}, dividing the proof into two steps. The first step consists in proving an estimate using the a priori assumption $V_p\left(Du\right)\in W^{1, 2}_{\loc}\left(\Omega\right)$.\\
In the second step, we use an approximation argument, considering a regularized version of the functional to whose minimizer we can apply the a priori estimate. Then we conclude by proving that such estimate is preserved in passing to the limit.\\
Before entering into the details of the proof, we want to stress that the necessity to use an approximation procedure is due to the assumptions on the function $g$ and on the datum $f$. If we had $f\in L^\infty_{\loc}\left(\Omega\right)$ and $g\in L^\infty_{\loc}\left(\Omega\right)$, it would be sufficient to apply the difference-quotient method to get $V_p\left(Du\right)\in W^{1,2}_\loc\left(\Omega\right)$ (see, for example \cite{AF} and \cite{Tolksdorff}).

\begin{proof}[Proof of Theorem \ref{CGPThm1}]
{\bf Step 1: the a priori estimate.}\\
	Our first step consists in proving that, if $u\in W^{1, p}_{\loc}\left(\Omega, \R^N\right)$ is a local minimizer of $\mathcal{F}$ such that
	$$
	V_p\left(Du\right)\in W^{1, 2}_{\loc}\left(\Omega\right),
	$$
	estimate \eqref{mainestimateVp} holds.\\
	Since $u\in W^{1, p}_{\loc}\left(\Omega, \R^N\right)$ is a local minimizer of $\mathcal{F}$, it solves the corresponding Euler-Lagrange system, that is, for any $\psi\in C^{\infty}_{0}\left(\Omega,\R^{N}\right)$, we have
\begin{equation}\label{EL}
	\int_{\Omega}\langle D_\xi F\left(x,Du(x)\right),D\psi(x)\rangle dx=\int_{\Omega}f(x)\cdot\psi(x).
\end{equation}
Let us fix a ball $B_{R}\Subset \Omega$ and arbitrary radii $\frac{R}{2}\le r<\tilde{s}<t<\tilde{t}<\lambda r<R,$ with $1<\lambda<2$. Let us consider a cut off function $\eta\in C^\infty_0\left(B_t\right)$ such that $\eta\equiv 1$ on
	$B_{\tilde{s}}$, $\left|D \eta\right|\le \frac{c}{t-\tilde{s}}$ and $\left|D^2 \eta\right|\le \frac{c}{\left(t-\tilde{s}\right)^2}$. From now on, with no	loss of generality, we suppose $R<1$.
For $\left|h\right|$ sufficiently small, we can choose, for any $s=1, ..., n$

$$\psi=\tau_{s, -h}\left(\eta^2\tau_{s, h}u\right)$$

as a test function in \eqref{EL}, and by Proposition \ref{findiffpr}, we get
\begin{eqnarray*}
	&&\int_\Omega \left<\tau_{s, h}D_\xi F\left(x, Du(x)\right), D\left(\eta^2(x)\tau_{s, h}u(x)\right)\right>dx\cr\cr
	&=&\int_\Omega f(x)\cdot\tau_{s, -h}\left(\eta^2(x)\tau_{s, h}u(x)\right)dx, 
\end{eqnarray*}
that is 
\begin{eqnarray*}
	I_1&=&\int_\Omega \left<D_\xi F\left(x+he_s, Du\left(x+he_s\right)\right)-D_\xi F\left(x+he_s, Du(x)\right), \eta^2(x)\tau_{s, h}Du(x)\right>dx\cr\cr
	&=&-\int_\Omega \left<D_\xi F\left(x+he_s, Du(x)\right)-D_\xi F\left(x, Du(x)\right), \eta^2(x)\tau_{s, h}Du(x)\right>dx\cr\cr
	&&-2\int_\Omega \left<D_\xi F\left(x+he_s, Du\left(x+he_s\right)\right)-D_\xi F\left(x, Du(x)\right), \eta(x)D\eta(x)\otimes\tau_{s, h}u(x)\right>dx\cr\cr
	&&+\int_\Omega f(x)\cdot\tau_{s, -h}\left(\eta^2(x)\tau_{s, h}u(x)\right)dx\cr\cr
	&:=&-I_2-I_3+I_4. 
\end{eqnarray*}
 Therefore

\begin{equation}\label{fullestimate}
	I_1\le\left|I_2\right|+\left|I_3\right|+\left|I_4\right|.
\end{equation}

By assumption \eqref{F2}, we get

\begin{equation}\label{I_1}
	I_1\ge\nu\int_\Omega\eta^2(x)\left(\mu^2+\left|Du(x)\right|^2+\left|Du\left(x+he_s\right)\right|^2\right)^\frac{p-2}{2}\left|\tau_{s, h}Du(x)\right|^2dx.
\end{equation}

For what concerns the term $I_2$, by \eqref{F4} and Young's Inequality with exponents $\left(2, 2\right)$, for any $\varepsilon>0$, we have

\begin{eqnarray*}\label{I_2*}
	\left|I_2\right|&\le&\left|h\right|\int_\Omega\eta^2(x)\left(g(x)+g(x+he_s)\right)\left(\mu^2+\left|Du(x)\right|^2+\left|Du\left(x+he_s\right)\right|^2\right)^\frac{p-1}{2}\left|\tau_{s, h}Du(x)\right|dx\cr\cr
	&\le&\varepsilon\int_\Omega\eta^2(x)\left(\mu^2+\left|Du(x)\right|^2+\left|Du\left(x+he_s\right)\right|^2\right)^\frac{p-2}{2}\left|\tau_{s, h}Du(x)\right|^2dx\cr\cr
	&&+c_\varepsilon\left|h\right|^2\int_{\Omega}\eta^2(x)\left(g(x)+g(x+he_s)\right)^2\left(\mu^2+\left|Du(x)\right|^2+\left|Du\left(x+he_s\right)\right|^2\right)^\frac{p}{2}dx.
\end{eqnarray*}

Now, by the assumption $g\in L^n_\loc\left(\Omega\right)$, we can use H\"older's inequality with exponents $\left(\frac{n}{2}, \frac{n}{n-2}\right)$ and by the properties of $\eta$ and Lemma \ref{le1}, we get

\begin{eqnarray}\label{I_2}
	\left|I_2\right|&\le&\varepsilon\int_\Omega\eta^2(x)\left(\mu^2+\left|Du(x)\right|^2+\left|Du\left(x+he_s\right)\right|^2\right)^\frac{p-2}{2}\left|\tau_{s, h}Du(x)\right|^2dx\cr\cr
	&&+c_\varepsilon\left|h\right|^2\left(\int_{B_t}\left(\mu^2+\left|Du(x)\right|^2+\left|Du\left(x+he_s\right)\right|^2\right)^\frac{np}{2\left(n-2\right)}dx\right)^\frac{n-2}{n}\cr\cr
	&&\cdot\left(\int_{B_t}\left(g(x)+g(x+he_s)\right)^{n}dx\right)^\frac{2}{n}\cr\cr
	&\le&\varepsilon\int_\Omega\eta^2(x)\left(\mu^2+\left|Du(x)\right|^2+\left|Du\left(x+he_s\right)\right|^2\right)^\frac{p-2}{2}\left|\tau_{s, h}Du(x)\right|^2dx\cr\cr
	&&+c_\varepsilon\left|h\right|^2\left(\int_{B_{\lambda r}}\left(\mu^2+\left|Du(x)\right|^2\right)^\frac{np}{2\left(n-2\right)}dx\right)^\frac{n-2}{n}\cdot\left(\int_{B_{\lambda r}}g^{n}(x)dx\right)^\frac{2}{n}.
\end{eqnarray}

Let us consider, now, th term $I_3$. We have 

\begin{eqnarray*}
	I_3&=&2\int_\Omega \left<\tau_{s, h}\left[D_\xi F\left(x, Du(x)\right)\right], \eta(x)D\eta(x)\otimes\tau_{s, h}u(x)\right>dx\cr\cr
	&=&2\int_\Omega \left<D_\xi F\left(x, Du(x)\right), \tau_{s, -h}\left[\eta(x)D\eta(x)\otimes\tau_{s, h}u(x)\right]\right>dx,
\end{eqnarray*}

so, by \eqref{F1}, we get

\begin{eqnarray}\label{I_3*}
	\left|I_3\right|&\le&c\int_\Omega \left(\mu^2+\left|Du(x)\right|^2\right)^\frac{p-1}{2}\left|\tau_{s, -h}\left[\eta(x)D\eta(x)\otimes\tau_{s, h}u(x)\right]\right|dx
\end{eqnarray}
and since, for any $x\in\supp(\eta)$ such that $x+he_s, x-he_s\in\supp(\eta)$, recaling the properties of $\eta$, we have

\begin{eqnarray}\label{tau1}
	\left|\tau_{s, -h}\left[\eta(x)D\eta(x)\otimes\tau_{s, h}u(x)\right]\right|&\le&\left|\tau_{s, -h}\eta(x)\cdot D\eta\left(x-he_s\right)\otimes\tau_{s, h}u\left(x-he_s\right)\right|\cr\cr
	&&+\left|\eta(x)\tau_{s, -h}D\eta(x)\otimes\tau_{s, h}u\left(x-he_s\right)\right|\cr\cr
	&&+\left|\eta(x)D\eta(x)\otimes\tau_{s, -h}\tau_{s, h}u(x)\right|\cr\cr
	&\le&\frac{c\left|h\right|}{\left(t-\tilde{s}\right)^2}\left|\tau_{s, h}u\left(x-he_s\right)\right|\cr\cr
	&&+\frac{c\left|h\right|}{t-\tilde{s}}\eta(x)\left|\tau_{s, -h}\tau_{s,h}u(x)\right|.
\end{eqnarray}

Inserting \eqref{tau1} into \eqref{I_3*}, we get

\begin{eqnarray}\label{I_3**}
	\left|I_3\right|&\le&\frac{c\left|h\right|}{\left(t-\tilde{s}\right)^2}\int_{B_{t}}\left(\mu^2+\left|Du(x)\right|^2\right)^\frac{p-1}{2}\left|\tau_{s, h}u\left(x-he_s\right)\right|dx\cr\cr
	&&+\frac{c\left|h\right|}{t-\tilde{s}}\int_{\Omega}\eta(x)\left(\mu^2+\left|Du(x)\right|^2\right)^\frac{p-1}{2}\left|\tau_{s, -h}\tau_{s, h}u(x)\right|dx,
\end{eqnarray}

and by H\"older's Inequality with exponents $\left(p, \frac{p}{p-1}\right)$ and the properties of $\eta$, \eqref{I_3**} becomes 
\begin{eqnarray*}
	\left|I_3\right|&\le&\frac{c\left|h\right|}{\left(t-\tilde{s}\right)^2}\left(\int_{B_{t}}\left(\mu^2+\left|Du(x)\right|^2\right)^\frac{p}{2}dx\right)^\frac{p-1}{p}\left(\int_{B_{t}}\left|\tau_{s, h}u\left(x-he_s\right)\right|^pdx\right)^\frac{1}{p}\cr\cr
	&&+\frac{c\left|h\right|}{t-\tilde{s}}\left(\int_{B_t}\left(\mu^2+\left|Du(x)\right|^2+\left|Du\left(x+he_s\right)\right|^2\right)^\frac{p}{2}dx\right)^\frac{p-1}{p}\cdot\left(\int_{B_t}\left|\tau_{s, -h}\tau_{s, h}u(x)\right|^pdx\right)^\frac{1}{p}.
\end{eqnarray*}
Now, by virtue of Lemma \ref{le1}, and using \eqref{lemma6GPestimate1}, we get

\begin{eqnarray}\label{I_3'}
	\left|I_3\right|&\le&\frac{c\left|h\right|^2}{\left(t-\tilde{s}\right)^2}\int_{B_{\tilde{t}}}\left(\mu^2+\left|Du(x)\right|^2\right)^\frac{p}{2}dx\cr\cr
	&&+\frac{c\left|h\right|^2}{t-\tilde{s}}\left(\int_{B_t}\left(\mu^2+\left|Du(x)\right|^2+\left|Du\left(x+he_s\right)\right|^2\right)^\frac{p}{2}dx\right)^\frac{p-1}{p}\cr\cr
	&&\cdot\left(\int_{B_{\tilde{t}}}\left|\tau_{s, h}Du(x)\right|^pdx\right)^\frac{1}{p}\cr\cr
	&\le&\frac{c\left|h\right|^2}{\left(t-\tilde{s}\right)^2}\int_{B_{\tilde{t}}}\left(\mu^2+\left|Du(x)\right|^2\right)^\frac{p}{2}dx\cr\cr
	&&+\frac{c\left|h\right|^2}{t-\tilde{s}}\left(\int_{B_t}\left(\mu^2+\left|Du(x)\right|^2+\left|Du\left(x+he_s\right)\right|^2\right)^\frac{p}{2}dx\right)^\frac{p-1}{p}\cr\cr
	&&\cdot\left(\int_{B_{\tilde{t}}}\left|\tau_{s, h}V_p\left(Du(x)\right)\right|^p\cdot\left(\mu^2+\left|Du(x)\right|^2+\left|Du\left(x+he_s\right)\right|^2\right)^{\frac{p\left(2-p\right)}{4}}dx\right)^\frac{1}{p}\cr\cr
	&\le&\frac{c\left|h\right|^2}{\left(t-\tilde{s}\right)^2}\int_{B_{\tilde{t}}}\left(\mu^2+\left|Du(x)\right|^2\right)^\frac{p}{2}dx+\frac{c\left|h\right|^2}{t-\tilde{s}}\left(\int_{B_{{\tilde{t}}}}\left(\mu^2+\left|Du(x)\right|^2\right)^{\frac{p}{2}}dx\right)^\frac{1}{2}\cr\cr
	&&\cdot\left(\int_{B_{\tilde{t}}}\left|\tau_{s, h}V_p\left(Du(x)\right)\right|^2dx\right)^\frac{1}{2}, 
\end{eqnarray}
where, in the last line, we used H\"{o}lder's inequality with exponents $\left(\frac{2}{p}, \frac{2}{2-p}\right)$.\\
Now, using Young's Inequality with exponents $\left(2, 2\right)$ and since $t-\tilde{s}<1$ and $t<\tilde{t}<\lambda r<R$, \eqref{I_3'} gives

\begin{equation}\label{I_3}
	\left|I_3\right|\le\sigma\int_{B_{\tilde{t}}}\left|\tau_{s, h}V_p\left(Du(x)\right)\right|^2dx+\frac{c_\sigma\left|h\right|^2}{\left(t-\tilde{s}\right)^2}\int_{B_{R}} \left(\mu^2+\left|Du\left(x\right)\right|^2\right)^\frac{p}{2}dx,
\end{equation}
for any $\sigma>0$.
For what concerns the term $I_4$, by virtue of Proposition \ref{findiffpr}, we have

\begin{eqnarray}\label{I_4'}
	I_4&=&\int_\Omega\eta^2\left(x\right)f(x)\tau_{s, -h}\left(\tau_{s, h}u(x)\right)dx\cr\cr
	&&+\int_\Omega\left[\eta\left(x-he_s\right)+\eta(x)\right]f(x)\tau_{s, -h}\eta(x)\tau_{s, h}u\left(x-he_s\right)dx\cr\cr
	&=:&J_1+J_2,
\end{eqnarray}

which yields 

\begin{equation}\label{I_4*}
	\left|I_4\right|\le\left|J_1\right|+\left|J_2\right|
\end{equation}

In order to estimate the term $J_1$, let us recall that, by virtue of the a priori assumption $V_p\left(Du\right)\in W^{1, 2}_\loc\left(\Omega\right)$ and Sobolev's embedding theorem, we have
$Du\in L^{\frac{np}{n-2}}_\loc\left(\Omega\right)$, which implies $Du\in L^{\frac{np}{n-2+p}}_\loc\left(\Omega\right)$. So, using H\"older's inequality with exponents $\left(\frac{np}{n\left(p-1\right)+2-p}, \frac{np}{n-2+p}\right)$, the properties of $\eta$ and Lemma \ref{le1}, we get

\begin{eqnarray}\label{J_1*}
	\left|J_1\right|&\le&\left(\int_{B_{t}}\left|f(x)\right|^\frac{np}{n\left(p-1\right)+2-p}dx\right)^\frac{n\left(p-1\right)+2-p}{np}\cdot\left(\int_{B_{t}}\left|\tau_{s, -h}\tau_{s, h}u(x)\right|^\frac{np}{n-2+p}dx\right)^\frac{n-2+p}{np}\cr\cr
	&\le&\left|h\right|\left(\int_{B_{t}}\left|f(x)\right|^\frac{np}{n\left(p-1\right)+2-p}dx\right)^\frac{n\left(p-1\right)+2-p}{np}\cdot\left(\int_{B_{\tilde{t}}}\left|\tau_{s, h}Du(x)\right|^\frac{np}{n-2+p}dx\right)^\frac{n-2+p}{np}
\end{eqnarray}

To go further, let us consider the second integral in \eqref{J_1*}. Using \eqref{lemma6GPestimate1}, we get
\begin{eqnarray*}
	\int_{B_{\tilde{t}}}\left|\tau_{s, h}Du(x)\right|^\frac{np}{n-2+p}dx&\le&\int_{B_{\tilde{t}}}\left|\tau_{s, h}V_p\left(Du(x)\right)\right|^\frac{np}{n-2+p}\cr\cr
	&&\cdot\left(\mu^2+\left|Du(x)\right|^2+\left|Du\left(x+he_s\right)\right|^2\right)^{\frac{2-p}{4}\cdot\frac{np}{n-2+p}}dx, 
\end{eqnarray*}
and, as long as $1<p<2$, we can use H\"older's inequality with exponents $\left(\frac{2\left(n-2+p\right)}{np}, \frac{2\left(n-2+p\right)}{\left(n-2\right)\left(2-p\right)}\right)$, thus getting

\begin{eqnarray}\label{VpJ1}
	\int_{B_{\tilde{t}}}\left|\tau_{s, h}Du(x)\right|^\frac{np}{n-2+p}dx&\le&\left(\int_{B_{\tilde{t}}}\left|\tau_{s, h}V_p\left(Du(x)\right)\right|^2dx\right)^\frac{np}{2\left(n-2+p\right)}\cr\cr
	&&\cdot\left(\int_{B_{\tilde{t}}}\left(\mu^2+\left|Du(x)\right|^2+\left|Du\left(x+he_s\right)\right|^2\right)^{\frac{np}{2\left(n-2\right)}}dx\right)^\frac{\left(n-2\right)\left(2-p\right)}{2\left(n-2+p\right)}.
\end{eqnarray}

Inserting \eqref{VpJ1} into \eqref{J_1*}, and using Young's inequality with exponents $\left(2, \frac{np}{n\left(p-1\right)+2-p}, \frac{2np}{\left(n-2\right)\left(2-p\right)}\right)$, we obtain

\begin{eqnarray}\label{J_1**}
	\left|J_1\right|&\le&\left|h\right|\left(\int_{B_{t}}\left|f(x)\right|^\frac{np}{n\left(p-1\right)+2-p}dx\right)^\frac{n\left(p-1\right)+2-p}{np}
	\cdot\left(\int_{B_{\tilde{t}}}\left|\tau_{s, h}V_p\left(Du(x)\right)\right|^2dx\right)^\frac{1}{2}\cr\cr
	&&\cdot\left(\int_{B_{\tilde{t}}}\left(\mu^2+\left|Du(x)\right|^2+\left|Du\left(x+he_s\right)\right|^2\right)^{\frac{np}{2\left(n-2\right)}}dx\right)^\frac{\left(n-2\right)\left(2-p\right)}{2np}\cr\cr
	&\le&c_\sigma\left|h\right|^2\int_{B_{t}}\left|f(x)\right|^\frac{np}{n\left(p-1\right)+2-p}dx
	\cr\cr
	&&+\sigma\left|h\right|^2\int_{B_{\tilde{t}}}\left(\mu^2+\left|Du(x)\right|^2+\left|Du\left(x+he_s\right)\right|^2\right)^{\frac{np}{2\left(n-2\right)}}dx\cr\cr
	&&+\sigma\int_{B_{\tilde{t}}}\left|\tau_{s, h}V_p\left(Du(x)\right)\right|^2dx.
\end{eqnarray}

Recalling that $t<\tilde{t}<\lambda r<R$ and by virtue of Lemma \ref{le1}, \eqref{J_1**} implies

\begin{eqnarray}\label{J_1}
	\left|J_1\right|&\le&c_\sigma\left|h\right|^2\int_{B_{R}}\left|f(x)\right|^\frac{np}{n\left(p-1\right)+2-p}dx
	\cr\cr
	&&+\sigma\left|h\right|^2\int_{B_{\lambda r}}\left(\mu^2+\left|Du(x)\right|^2\right)^{\frac{np}{2\left(n-2\right)}}dx\cr\cr
	&&+\sigma\int_{B_{{\tilde{t}}}}\left|\tau_{s, h}V_p\left(Du(x)\right)\right|^2dx.
\end{eqnarray}

For what concerns the term $J_2$, by virtue of the properties of $\eta$, we have

\begin{eqnarray}\label{J_2*}
	\left|J_2\right|&\le&c\int_{B_t}\left|f(x)\right|\left|\tau_{s, -h}\eta(x)\right|\left|\tau_{s, h}u\left(x-he_s\right)\right|dx\cr\cr
	&\le&\left|h\right|\left\Arrowvert D\eta\right\Arrowvert_{L^{\infty}\left(B_t\right)}\int_{B_t} \left|f(x)\right|\left|\tau_{s, h}u\left(x-he_s\right)\right|dx\cr\cr
	&\le&\frac{c\left|h\right|}{t-\tilde{s}}\int_{B_t} \left|f(x)\right|\left|\tau_{s, h}u\left(x-he_s\right)\right|dx.
\end{eqnarray}

Now, if we apply H\"older's and Young's Inequality in \eqref{J_2*} with exponents $\left(\frac{np}{n\left(p-1\right)+2}, \frac{np}{n-2}\right)$, we get

\begin{eqnarray}\label{J_2}
	\left|J_2\right|&\le&\frac{c\left|h\right|}{t-\tilde{s}}\left(\int_{B_t} \left|f(x)\right|^\frac{np}{n\left(p-1\right)+2}dx\right)^\frac{n\left(p-1\right)+2}{np}\cdot\left(\int_{B_t}\left|\tau_{s, h}u\left(x-he_s\right)\right|^\frac{np}{n-2}dx\right)^\frac{n-2}{np}\cr\cr
	&\le&\frac{c_\sigma\left|h\right|^2}{\left(t-\tilde{s}\right)^\frac{np}{n\left(p-1\right)+2}}\int_{B_t} \left|f(x)\right|^\frac{np}{n\left(p-1\right)+2}dx+\sigma\left|h\right|^2\int_{B_{\lambda r}}\left|Du\left(x\right)\right|^\frac{np}{n-2}dx,
\end{eqnarray} 
where we also used Lemma \ref{le1}, since $Du\in L^{\frac{np}{n-2}}_{\loc}\left(\Omega\right)$.\\

By virtue of \eqref{J_1} and \eqref{J_2}, \eqref{I_4*} gives

\begin{eqnarray}\label{I_4}
	\left|I_4\right|&\le&c_\sigma\left|h\right|^2\int_{B_{R}}\left|f(x)\right|^\frac{np}{n\left(p-1\right)+2-p}dx+\frac{c_\sigma\left|h\right|^2}{\left(t-\tilde{s}\right)^\frac{np}{n\left(p-1\right)+2}}\int_{B_t} \left|f(x)\right|^\frac{np}{n\left(p-1\right)+2}dx
	\cr\cr
	&&+2\sigma\left|h\right|^2\int_{B_{\lambda r}}\left(\mu^2+\left|Du(x)\right|^2\right)^{\frac{np}{2\left(n-2\right)}}dx+\sigma\int_{B_{{\tilde{t}}}}\left|\tau_{s, h}V_p\left(Du(x)\right)\right|^2dx.
\end{eqnarray}

Inserting \eqref{I_1}, \eqref{I_2}, \eqref{I_3}, and \eqref{I_4} into \eqref{fullestimate}, and choosing $\varepsilon<\frac{\nu}{2}$, we get

\begin{eqnarray*}\label{fullestimate1}
	&&c\left(\nu\right)\int_\Omega\eta^2(x)\left(\mu^2+\left|Du(x)\right|^2+\left|Du\left(x+he_s\right)\right|^2\right)^\frac{p-2}{2}\left|\tau_{s, h}Du(x)\right|^2dx\cr\cr
	&\le&c\left|h\right|^2\left(\int_{B_{\lambda r}}\left(\mu^2+\left|Du(x)\right|^2\right)^\frac{np}{2\left(n-2\right)}dx\right)^\frac{n-2}{n}\cdot\left(\int_{B_{\lambda r}}g^{n}(x)dx\right)^\frac{2}{n}\cr\cr
	&&+2\sigma\int_{B_{\tilde{t}}}\left|\tau_{s, h}V_p\left(Du(x)\right)\right|^2dx+\frac{c_\sigma\left|h\right|^2}{\left(t-\tilde{s}\right)^2}\int_{B_{R}} \left(\mu^2+\left|Du\left(x\right)\right|^2\right)^\frac{p}{2}dx\cr\cr
	&&+c_\sigma\left|h\right|^2\int_{B_{R}}\left|f(x)\right|^\frac{np}{n\left(p-1\right)+2-p}dx+\frac{c_\sigma\left|h\right|^2}{\left(t-\tilde{s}\right)^\frac{np}{n\left(p-1\right)+2}}\int_{B_t} \left|f(x)\right|^\frac{np}{n\left(p-1\right)+2}dx
	\cr\cr
	&&+2\sigma\left|h\right|^2\int_{B_{\lambda r}}\left(\mu^2+\left|Du(x)\right|^2\right)^{\frac{np}{2\left(n-2\right)}}dx,
\end{eqnarray*}

for $\sigma>0$ that will be chosen later.\\ 
So, by \eqref{lemma6GPestimate1} and the properties of $\eta$, we have

\begin{eqnarray*}\label{tauhVp*}
 	&&\int_{B_{\tilde{s}}}\left|\tau_{s, h}V_p\left(Du(x)\right)\right|^2dx\cr\cr
 	&\le&3\sigma\left|h\right|^2\int_{B_{\lambda r}}\left(\mu^2+\left|Du(x)\right|^2\right)^\frac{np}{2\left(n-2\right)}dx+c_\sigma\left|h\right|^2\int_{B_{\lambda r}}g^{n}(x)dx\cr\cr
 	&&+2\sigma\int_{B_{\tilde{t}}}\left|\tau_{s, h}V_p\left(Du(x)\right)\right|^2dx+\frac{c_\sigma\left|h\right|^2}{\left(t-\tilde{s}\right)^2}\int_{B_{R}} \left(\mu^2+\left|Du\left(x\right)\right|^2\right)^\frac{p}{2}dx\cr\cr
 	&&+c_\sigma\left|h\right|^2\int_{B_{R}}\left|f(x)\right|^\frac{np}{n\left(p-1\right)+2-p}dx+\frac{c_\sigma\left|h\right|^2}{\left(t-\tilde{s}\right)^\frac{np}{n\left(p-1\right)+2}}\int_{B_t} \left|f(x)\right|^\frac{np}{n\left(p-1\right)+2}dx,
 \end{eqnarray*}

where we also used Young's Inequality with exponents $\left(\frac{n}{2}, \frac{n}{n-2}\right)$.\\
Now, Lemma \ref{le1} implies

\begin{eqnarray}\label{tauhVp}
	&&\int_{B_{\tilde{s}}}\left|\tau_{s, h}V_p\left(Du(x)\right)\right|^2dx\cr\cr
	&\le&3\sigma\left|h\right|^2\int_{B_{\lambda r}}\left(\mu^2+\left|Du(x)\right|^2\right)^\frac{np}{2\left(n-2\right)}dx+c\cdot\sigma\left|h\right|^2\int_{B_{\lambda r}}\left|DV_p\left(Du(x)\right)\right|^2dx\cr\cr
	&&+\frac{c_\sigma\left|h\right|^2}{\left(t-\tilde{s}\right)^2}\int_{B_{R}} \left(\mu^2+\left|Du\left(x\right)\right|^2\right)^\frac{p}{2}dx+c_\sigma\left|h\right|^2\int_{B_{\lambda r}}g^{n}(x)dx\cr\cr
	&&+c_\sigma\left|h\right|^2\int_{B_{R}}\left|f(x)\right|^\frac{np}{n\left(p-1\right)+2-p}dx+\frac{c_\sigma\left|h\right|^2}{\left(t-\tilde{s}\right)^\frac{np}{n\left(p-1\right)+2}}\int_{B_t} \left|f(x)\right|^\frac{np}{n\left(p-1\right)+2}dx.
\end{eqnarray}

Since \eqref{tauhVp} holds for any $s=1,..., n$ and, by virtue of the a priori assumption, $V_p\left(Dv\right)\in W^{1,2}_\loc\left(\Omega\right)$ and by Lemma \ref{Giusti8.2}, we get

\begin{eqnarray*}
	&&\int_{B_{\tilde{s}}}\left|DV_p\left(Du(x)\right)\right|^2dx\cr\cr
	&\le&c\cdot\sigma\int_{B_{{\lambda r}}}\left|DV_p\left(Du(x)\right)\right|^2+3\sigma\int_{B_{\lambda r}}\left(\mu^2+\left|Du(x)\right|^2\right)^\frac{np}{2\left(n-2\right)}dx\cr\cr
	&&+\frac{c_\sigma}{\left(t-\tilde{s}\right)^2}\int_{B_{R}} \left(\mu^2+\left|Du\left(x\right)\right|^2\right)^\frac{p}{2}dx+c_\sigma\int_{B_{\lambda r}}g^{n}(x)dx\cr\cr
	&&+c_\sigma\int_{B_{R}}\left|f(x)\right|^\frac{np}{n\left(p-1\right)+2-p}dx
	+\frac{c_\sigma}{\left(t-\tilde{s}\right)^\frac{np}{n\left(p-1\right)+2}}\int_{B_R} \left|f(x)\right|^\frac{np}{n\left(p-1\right)+2}dx,
\end{eqnarray*}
and since $t-\tilde{s}<1$, setting
\begin{equation}\label{beta}
\beta\left(n, p\right)=\max\Set{2, \frac{np}{n\left(p-1\right)+2}}, 
\end{equation}
we get
\begin{eqnarray}\label{beforeIterThm1*}
		&&\int_{B_{\tilde{s}}}\left|DV_p\left(Du(x)\right)\right|^2dx\cr\cr
	&\le&c\cdot\sigma\int_{B_{{\lambda r}}}\left|DV_p\left(Du(x)\right)\right|^2+3\sigma\int_{B_{\lambda r}}\left(\mu^2+\left|Du(x)\right|^2\right)^\frac{np}{2\left(n-2\right)}dx\cr\cr
	&&+\frac{c_\sigma}{\left(t-\tilde{s}\right)^{\beta\left(n, p\right)}}\left[\int_{B_{R}} \left(\mu^2+\left|Du\left(x\right)\right|^2\right)^\frac{p}{2}dx+\int_{B_{\lambda r}}g^{n}(x)dx\right.\cr\cr
	&&\left.+\int_{B_{R}}\left|f(x)\right|^\frac{np}{n\left(p-1\right)+2-p}dx
	+\int_{B_R} \left|f(x)\right|^\frac{np}{n\left(p-1\right)+2}dx\right].
\end{eqnarray}
Now let us notice that, since $\frac{np}{n\left(p-1\right)+2}<\frac{np}{n\left(p-1\right)+2-p}$, we have

\begin{eqnarray}\label{immersion1}
	\int_{B_R} \left|f(x)\right|^\frac{np}{n\left(p-1\right)+2}dx\le c\int_{B_R} \left|f(x)\right|^\frac{np}{n\left(p-1\right)+2-p}dx+c\left|B_R\right|.
\end{eqnarray}
Plugging \eqref{immersion1} into \eqref{beforeIterThm1*}, we get

\begin{eqnarray}\label{beforeIterThm1**}
	&&\int_{B_{\tilde{s}}}\left|DV_p\left(Du(x)\right)\right|^2dx\cr\cr
	&\le&c\cdot\sigma\int_{B_{{\lambda r}}}\left|DV_p\left(Du(x)\right)\right|^2+3\sigma\int_{B_{\lambda r}}\left(\mu^2+\left|Du(x)\right|^2\right)^\frac{np}{2\left(n-2\right)}dx\cr\cr
	&&+\frac{c_\sigma}{\left(t-\tilde{s}\right)^{\beta\left(n, p\right)}}\left[\int_{B_{R}} \left(\mu^2+\left|Du\left(x\right)\right|^2\right)^\frac{p}{2}dx+\int_{B_{\lambda r}}g^{n}(x)dx\right.\cr\cr
	&&\left.+\int_{B_R} \left|f(x)\right|^\frac{np}{n\left(p-1\right)+2-p}dx+\left|B_R\right|\right].
\end{eqnarray}

Moreover, applying Sobolev's inequality to the function $V_p\left(Du\right)$ and recalling \eqref{stinorma}, for a positive constant $c=c(n, p)$ we get

\begin{eqnarray}\label{V_pSobolev}
	\int_{B_{\lambda r}}\left|Du(x)\right|^\frac{np}{n-2}dx&\le&	c\int_{B_{\lambda r}}\left|V_p\left(Du(x)\right)\right|^\frac{2n}{n-2}dx+c\mu^\frac{np}{n-2}\left|B_R\right|\cr\cr
	&\le&c\int_{B_{\lambda r}}\left|DV_p\left(Du(x)\right)\right|^2dx+c\int_{B_{\lambda r}}\left|V_p\left(Du(x)\right)\right|^2dx\cr\cr
	&&+c\left|B_R\right|\cr\cr
	&\le&c\int_{B_{\lambda r}}\left|DV_p\left(Du(x)\right)\right|^2dx+c\int_{B_R}\left(\mu^2+\left|Du(x)\right|^2\right)^\frac{p}{2}dx\cr\cr
	&&+c\left|B_R\right|,
\end{eqnarray}

where we also used the fact that $\mu\in[0, 1]$.\\
Now, plugging \eqref{V_pSobolev} into \eqref{beforeIterThm1**}, and recalling that $t-\tilde{s}<1$ and $\lambda r<R$, we get

\begin{eqnarray*}
	\int_{B_{\tilde{s}}}\left|DV_p\left(Du(x)\right)\right|^2dx
	&\le&c\cdot\sigma\int_{B_{{\lambda r}}}\left|DV_p\left(Du(x)\right)\right|^2\cr\cr
	&&+\frac{c_\sigma}{\left(t-\tilde{s}\right)^{\beta\left(n, p\right)}}\left[\int_{B_{R}} \left(\mu^2+\left|Du\left(x\right)\right|^2\right)^\frac{p}{2}dx+\int_{B_{R}}g^{n}(x)dx\right.\cr\cr
	&&\left.+\int_{B_R} \left|f(x)\right|^\frac{np}{n\left(p-1\right)+2-p}dx+\left|B_R\right|\right].
\end{eqnarray*}
and choosing $\sigma>0$ such that 
$$
c\cdot\sigma=\frac{1}{2}, 
$$
we get
\begin{eqnarray}\label{DVpDu}
\int_{B_{\tilde{s}}}\left|DV_p\left(Du(x)\right)\right|^2dx&\le&\frac{1}{2}\int_{B_{{\lambda r}}}\left|DV_p\left(Du(x)\right)\right|^2\cr\cr
&&+\frac{c}{\left(t-\tilde{s}\right)^{\beta\left(n, p\right)}}\left[\int_{B_{R}} \left(\mu^2+\left|Du\left(x\right)\right|^2\right)^\frac{p}{2}dx+\int_{B_{R}}g^{n}(x)dx\right.\cr\cr
&&\left.+\int_{B_R} \left|f(x)\right|^\frac{np}{n\left(p-1\right)+2-p}dx+\left|B_R\right|\right].
\end{eqnarray}

Since \eqref{DVpDu} holds for any $\frac{R}{2}\le r<\tilde{s}<t<\lambda r<R,$ with $1<\lambda<2$, and the constant $c$ depends on $n, p, L, \nu, L_1, \ell$ but is independent of the radii, we can take the limit as $\tilde{s}\to r$ and $t\to\lambda r$, thus getting

\begin{eqnarray*}\label{lambdaDVpDu}
\int_{B_{r}}\left|DV_p\left(Du(x)\right)\right|^2dx&\le&\frac{1}{2}\int_{B_{\lambda r}}\left|DV_p\left(Du(x)\right)\right|^2dx\cr\cr
&&+\frac{c}{r^{\beta\left(n, p\right)}\left(\lambda-1\right)^{\beta\left(n, p\right)}}\left[\int_{B_{R}} \left(\mu^2+\left|Du\left(x\right)\right|^2\right)^\frac{p}{2}dx+\int_{B_{R}}g^{n}(x)dx\right.\cr\cr
&&\left.+\int_{B_R} \left|f(x)\right|^\frac{np}{n\left(p-1\right)+2-p}dx+\left|B_R\right|\right].
\end{eqnarray*}

Now, if we set

$$
h(r)=\int_{B_r}\left|DV_p\left(Du(x)\right)\right|^2dx,
$$

$$
A=c\left[\int_{B_{R}} \left(\mu^2+\left|Du\left(x\right)\right|^2\right)^\frac{p}{2}dx+\int_{B_{R}}g^{n}(x)dx+\int_{B_R} \left|f(x)\right|^\frac{np}{n\left(p-1\right)+2-p}dx+\left|B_R\right|\right]
$$

and

$$
B=0, 
$$

and apply Lemma \ref{iter} with 

$$
\theta=\frac{1}{2}\qquad\mbox{ and }\qquad \gamma=\beta\left(n, p\right), 
$$

we get

\begin{eqnarray}\label{aprioriestimate1*}
\int_{B_{\frac{R}{2}}}\left|DV_p\left(Du(x)\right)\right|^2dx&\le&\frac{c}{R^{\beta\left(n, p\right)}}\left[\int_{B_{R}} \left(\mu^2+\left|Du\left(x\right)\right|^2\right)^\frac{p}{2}dx\right.\cr\cr
&&\left.+\int_{B_R} \left|f(x)\right|^\frac{np}{n\left(p-1\right)+2-p}dx+\int_{B_{R}}g^{n}(x)dx+\left|B_R\right|\right],
\end{eqnarray}
that is the desired a priori estimate.
\medskip\\
{\bf Step 2: the approximation.}\\
Now we want to complete the proof of Theorem \ref{CGPThm1}, using the a priori estimate \eqref{aprioriestimate1*}, and a classical approximation argument.

Let us consider an open set $\Omega'\Subset\Omega$, and a function $\phi\in C^{\infty}_0(B_1(0))$ such that $0\le\phi\le1$ and $\int_{B_1(0)}\phi(x)dx=1$, and a standard family of mollifiers $\set{\phi_\varepsilon}_\varepsilon$  defined as follows

\begin{equation*}
	\phi_\varepsilon(x)=\frac{1}{\varepsilon^n}\phi\left(\frac{x}{\varepsilon}\right),
\end{equation*}

for any $\varepsilon\in\left(0, d\left(\Omega', \partial\Omega\right)\right)$, so that, for each $\varepsilon$, $\phi_\varepsilon\in C^{\infty}_0\left(B_\varepsilon(0)\right)$, $0\le\phi_\varepsilon\le1$, $\int_{B_\varepsilon(0)}\phi_\varepsilon(x)dx=1.$

It is well known that, for any $h\in L^1_{\loc}\left(\Omega'\right)$, setting 
\begin{equation*}
	h_\varepsilon(x)=h\ast\phi_\varepsilon(x)=\int_{B_\varepsilon}\phi_\varepsilon(y)h(x+y)dy=\int_{B_1}\phi(\omega)h(x+\varepsilon\omega)d\omega,
\end{equation*}
we have $h_\varepsilon\in C^\infty\left(\Omega'\right)$.

Let us fix a ball $B_{\tilde{R}}=B_{\tilde{R}}\left(x_0\right)\Subset\Omega'$, with $\tilde{R}<1$ and, for each $\varepsilon\in\left(0, d\left(\Omega', \partial\Omega\right)\right)$, let us consider the functional
\begin{equation*}
\mathcal{F}_\varepsilon\left(w, B_{\tilde{R}}\right)=\int_{ B_{\tilde{R}}}\left[ F_\varepsilon\left(x,Dw(x)\right)-f_\varepsilon(x)\cdot
		w(x)\right]dx,
\end{equation*}
where
\begin{equation}\label{Fepsdef}
	F_\varepsilon(x,\xi)=\int_{B_1}\phi(\omega)F(x+\varepsilon\omega, \xi)d\omega
\end{equation}
and
\begin{equation}\label{fepsdef}
	f_\varepsilon=f\ast\phi_\varepsilon.
\end{equation}

Let us recall that

\begin{equation}\label{convF}
		\int_{B_{\tilde{R}}}F_\varepsilon\left(x, \xi\right)dx\to\int_{B_{\tilde{R}}}F\left(x, \xi\right)dx,\qquad\mbox{ as }\varepsilon\to0
\end{equation}

for any $\xi\in\R^{n\times N}.$\\
Moreover, since $f\in L^{\frac{np}{n\left(p-1\right)+2-p}}_{\loc}\left(\Omega\right)$, we have

\begin{equation}\label{convf}
	f_\varepsilon\to f \qquad\mbox{ strongly in }L^{\frac{np}{n\left(p-1\right)+2-p}}\left(B_{\tilde{R}}\right),
\end{equation}

and since $\frac{np}{n\left(p-1\right)+2-p}>\frac{np}{n\left(p-1\right)+p}=\frac{p^*}{p^*-1}$, we also have

\begin{equation}\label{convfp*'}
	f_\varepsilon\to f \qquad\mbox{ strongly in }L^{\frac{p^*}{p^*-1}}\left(B_{\tilde{R}}\right),
\end{equation}
as $\varepsilon\to0.$\\

It is easy to check that \eqref{F1}--\eqref{F4} imply
\begin{equation}\label{F1eps}
	\ell \left(\mu^2+\left|\xi\right|^2\right)^\frac{p}{2}\le F_\varepsilon(x,\xi)\le L\left(\mu^2+\left|\xi\right|^2\right)^\frac{p}{2},
\end{equation}

\begin{equation}\label{F2eps}
	\langle D_{\xi}F_\varepsilon(x,\xi)-D_{\xi}F_\varepsilon(x,\eta),\xi-\eta\rangle \ge \nu\left(\mu^2+\left|\xi\right|^2+\left|\eta\right|^2\right)^\frac{p-2}{2}|\xi-\eta|^{2},
\end{equation}

\begin{equation}\label{F3eps}
	\left|D_{\xi}F_\varepsilon(x,\xi)-D_{\xi}F_\varepsilon(x,\eta)\right| \le L_1\left(\mu^2+\left|\xi\right|^2+\left|\eta\right|^2\right)^\frac{p-2}{2}|\xi-\eta|,
\end{equation}

\begin{equation}\label{F4eps}
	\left|D_{\xi}F_\varepsilon\left(x,\xi\right)-D_{\xi}F_\varepsilon\left(y,\xi\right)\right|\le\left(g_\varepsilon(x)+g_\varepsilon(y)\right)\left(\mu^2+\left|\xi\right|^2\right)^\frac{p-1}{2}\left|x-y\right|,
\end{equation}
for a.e. $x, y\in B_{\tilde{R}}$ and  every $\xi, \eta\in\R^{n\times N}$, where 
\begin{equation}\label{gepsdef}
	g_\varepsilon=g\ast\phi_\varepsilon.
\end{equation}
Since $g\in L^{n}_{\loc}\left(\Omega\right)$, we have
\begin{equation}\label{convg}
	g_\varepsilon\to g\qquad\mbox{ strongly in }L^n\left(B_{\tilde{R}}\right),\mbox{ as } \varepsilon\to0.
\end{equation}
For each $\varepsilon$, let $v_\varepsilon\in u+W^{1,p}_{0}(B_R)$ be the solution to
$$\min\Set{\mathcal{F}_\varepsilon\left(v,B_{\tilde{R}}\right): v\in u+W^{1,p}_{0}\left(B_{\tilde{R}}\right)},$$
where $u\in W^{1,p}_{\loc}\left(\Omega\right)$ is a local minimizer of \eqref{modenergy}.\\
By virtue of the minimality of $v_\varepsilon$, we have 
	\begin{equation*}
		\int_{B_{\tilde{R}}}\left[F_\varepsilon\left(x,Dv_\varepsilon(x)\right)-f_\varepsilon(x)\cdot v_\varepsilon(x)\right]dx\le\int_{B_{\tilde{R}}}\left[F_\varepsilon\left(x,Du(x)\right)-f_\varepsilon(x)\cdot u(x)\right]dx,
	\end{equation*} 
which means
\begin{equation*}
	\int_{B_{\tilde{R}}}F_\varepsilon\left(x,Dv_\varepsilon(x)\right)dx\le\int_{B_{\tilde{R}}}\left[F_\varepsilon\left(x,Du(x)\right)+f_\varepsilon(x)\cdot\left(v_\varepsilon(x)-u(x)\right)\right]dx,
\end{equation*} 
and by \eqref{F1eps} we get
\begin{eqnarray}\label{stima1*}
\ell\int_{B_{\tilde{R}}}\left(\mu^2+\left|Dv_\varepsilon(x)\right|^2\right)^\frac{p}{2}dx&\le&\int_{B_{\tilde{R}}}F_\varepsilon\left(x,Dv_\varepsilon(x)\right)dx\cr\cr
&\le&\int_{B_{\tilde{R}}}\left[F_\varepsilon\left(x,Du(x)\right)+f_\varepsilon(x)\cdot\left(v_\varepsilon(x)-u(x)\right)\right]dx\cr\cr
&\le&L\int_{B_{\tilde{R}}}\left(\mu^2+\left|Du(x)\right|^2\right)^\frac{p}{2}dx+\cr\cr
&&+\int_{B_{\tilde{R}}}\left|f_\varepsilon(x)\right|\left|v_\varepsilon(x)-u(x)\right|dx.
\end{eqnarray}

If we use H\"{o}lder's and Young's inequalities with exponents $\left(p^*, \frac{p^*}{p^*-1}\right)$ in \eqref{stima1*} and apply Sobolev's inequality to the function $v_\varepsilon-u\in W^{1, p}_0\left(B_{\tilde{R}}\right)$, for any $\sigma>0$, we get
 
\begin{eqnarray}\label{stima1**}
	&&\ell\int_{B_{\tilde{R}}}\left(\mu^2+\left|Dv_\varepsilon(x)\right|^2\right)^\frac{p}{2}dx\cr\cr
	&\le&L\int_{B_{\tilde{R}}}\left(\mu^2+\left|Du(x)\right|^2\right)^\frac{p}{2}dx+c_\sigma\int_{B_{\tilde{R}}}\left|f_\varepsilon(x)\right|^{\frac{p^*}{p^*-1}}dx\cr\cr
	&&+\sigma\int_{B_{\tilde{R}}}\left|v_\varepsilon(x)-u(x)\right|^{p^*}dx\cr\cr
	&\le&L\int_{B_{\tilde{R}}}\left(\mu^2+\left|Du(x)\right|^2\right)^\frac{p}{2}dx+c_\sigma\int_{B_{\tilde{R}}}\left|f_\varepsilon(x)\right|^{\frac{p^*}{p^*-1}}dx\cr\cr
	&&+\sigma\left(\int_{B_{\tilde{R}}}\left|Dv_\varepsilon(x)-Du(x)\right|^{p}dx\right)\cr\cr
	&\le&c_\sigma\int_{B_{\tilde{R}}}\left(\mu^2+\left|Du(x)\right|^2\right)^\frac{p}{2}dx+c_\sigma\int_{B_{\tilde{R}}}\left|f_\varepsilon(x)\right|^{\frac{p^*}{p^*-1}}dx\cr\cr
	&&+\sigma\left(\int_{B_{\tilde{R}}}\left(\mu^2+\left|Dv_\varepsilon(x)\right|^2\right)^\frac{p}{2}dx\right).
\end{eqnarray}

Now, if we choose $\sigma<\frac{\ell}{2}$ in \eqref{stima1**}, we have

\begin{eqnarray}\label{stima1***}
	&&\ell\int_{B_{\tilde{R}}}\left(\mu^2+\left|Dv_\varepsilon(x)\right|^2\right)^\frac{p}{2}dx\cr\cr
	&\le&L\int_{B_{\tilde{R}}}\left(\mu^2+\left|Du(x)\right|^2\right)^\frac{p}{2}dx+c\int_{B_{\tilde{R}}}\left|f_\varepsilon(x)\right|^{\frac{p^*}{p^*-1}}dx.
\end{eqnarray}

By virtue of \eqref{convfp*'}, \eqref{stima1***} implies that $\Set{v_\varepsilon}_\varepsilon$ is bounded in $W^{1, p}_{\loc}\left(B_{\tilde{R}}\right)$.
Therefore  there exists $v\in W^{1,p}\left(B_{\tilde{R}}\right)$ such that

	\begin{equation*}\label{convdebLp}
		v_\varepsilon\rightharpoonup v\qquad\mbox{ weakly in }W^{1,p}\left(B_{\tilde{R}}\right),
	\end{equation*} 

\begin{equation*}\label{convforteLp}
	v_\varepsilon\to v\qquad\mbox{ strongly in }L^{p}\left(B_{\tilde{R}}\right),
\end{equation*} 

and

\begin{equation*}\label{aeconv}
	v_\varepsilon\to v\qquad\mbox{ almost everywhere in }B_{\tilde{R}},
\end{equation*}

up to a subsequence, as $\varepsilon\to0$.\\
On the other hand, since $V_p\left(Dv_\varepsilon\right)\in W^{1,2}_{\loc}\left(B_{\tilde{R}}\right)$ and so $Dv_\varepsilon \in W^{1,p}_{\loc}\left(B_{\tilde{R}}\right)$ we are legitimated to apply estimates \eqref{aprioriestimate1*}, thus getting
\begin{eqnarray}\label{dersecapp1}
&&\int_{B_{\frac{r}{2}}}\left|DV_p\left(Dv_\varepsilon(x)\right)\right|^2dx\cr\cr
&\le&\frac{c}{r^{\beta\left(n, p\right)}}\left[\int_{B_{r}} \left(\mu^2+\left|Dv_\varepsilon\left(x\right)\right|^2\right)^\frac{p}{2}dx+\int_{B_R} \left|f_\varepsilon(x)\right|^\frac{np}{n\left(p-1\right)+2-p}dx\right.\cr\cr
&&\left.+c\int_{B_{r}}g^{n}_\varepsilon(x)dx+\left|B_r\right|^{\frac{n\left(p-1\right)+2}{np}}\right],
\end{eqnarray}
for any ball $B_r\Subset B_{\tilde{R}}$.

By virtue of \eqref{convf}, \eqref{convfp*'}, \eqref{convg} and \eqref{stima1***}, the right hand side of \eqref{dersecapp1} can be bounded independently of $\varepsilon$. For this reason, recalling Lemma \ref{differentiabilitylemma}, we also infer that, for each $\varepsilon$, $v_\varepsilon\in W^{2, p}_{\loc}\left(B_{\tilde{R}}\right)$, and recalling \eqref{differentiabilityestimate}, we also deduce that $\Set{v_\varepsilon}_\varepsilon$ is bounded in $W^{2, p}_{\loc}\left(B_{r}\right)$\\ 
Hence

\begin{equation*}\label{vconvdebWp}
	v_\varepsilon\rightharpoonup v\qquad\mbox{ weakly in } W^{2,p}\left(B_{r}\right),
\end{equation*}

\begin{equation}\label{vconvforW1p}
		v_\varepsilon\to v\qquad\mbox{ strongly in } W^{1,p}\left(B_{r}\right),
\end{equation}

and

\begin{equation}\label{aeconvDv}
	Dv_\varepsilon\to Dv
\qquad\mbox{ almost everywhere in }B_r,
\end{equation}

up to a subsequence, as $\varepsilon\to0$.\\
Moreover, by the continuity of $\xi\mapsto DV_p(\xi)$ and \eqref{aeconvDv}, we get $DV_p\left(Dv_\varepsilon\right)\to DV_p\left(Dv\right)$ almost everywhere, and since the right-hand side of \eqref{dersecapp1} can be bounded independently of $\varepsilon$, by Fatou's Lemma, passing to the limit as $\varepsilon\to0$ in \eqref{dersecapp1}, by \eqref{convf}, \eqref{convg} and \eqref{vconvforW1p}, we get

\begin{eqnarray}\label{dersecapp1*}
&&\int_{B_{\frac{r}{2}}}\left|DV_p\left(Dv(x)\right)\right|^2dx\cr\cr
&\le&\frac{c}{r^{\beta\left(n, p\right)}}\left[\int_{B_{r}} \left(\mu^2+\left|Dv\left(x\right)\right|^2\right)^\frac{p}{2}dx+\int_{B_R} \left|f(x)\right|^\frac{np}{n\left(p-1\right)+2-p}dx\right.\cr\cr
&&\left.+c\int_{B_{r}}g^{n}(x)dx+\left|B_r\right|^{\frac{n\left(p-1\right)+2}{np}}\right].
\end{eqnarray}
Our final step is to prove that $u=v$ a.e. in $B_{\tilde{R}}$.\\
First, let us observe that, using H\"older's inequality with exponents $\left(p^*, \frac{p^*}{p^*-1}\right)$, we get

\begin{eqnarray}\label{finalconv1*}
	&&\int_{B_{\tilde{R}}}\left|f_\varepsilon(x)-f(x)\right|\left|u(x)\right|dx\cr\cr
	&\le&\left(\int_{B_{\tilde{R}}}\left|u(x)\right|^{p^*}dx\right)^\frac{1}{{p^*}}\cdot\left(\int_{B_{\tilde{R}}}\left|f_\varepsilon(x)-f(x)\right|^{\frac{p^*}{p^*-1}}dx\right)^{\frac{p^*-1}{p^*}}, 
\end{eqnarray}
and recalling \eqref{convF} and \eqref{convfp*'}, \eqref{finalconv1*} implies

\begin{equation}\label{finalconv1}
\lim_{\varepsilon\to 0}\int_{B_{\tilde{R}}}\left[F_\varepsilon\left(x,Du(x)\right)-f_\varepsilon(x)\cdot
	u(x)\right]dx=\int_{B_{\tilde{R}}}\left[F\left(x,Du(x)\right)-f(x)\cdot
	u(x)\right].dx
\end{equation}

The minimality of $u$, Fatou's Lemma, the lower semicontinuity of $\mathcal{F}_\varepsilon$ and the minimality of $v_\varepsilon$ imply
	\begin{eqnarray*}
		&&\int_{B_{\tilde{R}}}\left[F\left(x,Du(x)\right)-f(x)\cdot
		u(x)\right]dx\cr\cr
		&\le&\int_{B_{\tilde{R}}}\left[F\left(x,Dv(x)\right)-f(x)\cdot
		v(x)\right]dx\cr\cr
		&\le&\liminf_{\varepsilon\to 0}\int_{B_{\tilde{R}}}\left[F\left(x,Dv_\varepsilon(x)\right)-f(x)\cdot
		v_\varepsilon(x)\right]dx\cr\cr
		&=&\liminf_{\varepsilon\to 0} \int_{B_{\tilde{R}}}\left[ F_\varepsilon\left(x,Dv_\varepsilon(x)\right)-f_\varepsilon(x)\cdot
		v_\varepsilon(x)\right]dx\cr\cr
		&\le&\liminf_{\varepsilon\to 0} \int_{B_{\tilde{R}}}\left[ F_\varepsilon\left(x,Du(x)\right)-f_\varepsilon(x)\cdot
		u(x)\right]dx\cr\cr
		&=&\int_{B_{\tilde{R}}}\left[F\left(x,Du(x)\right)-f(x)\cdot
		u(x)\right]dx,
	\end{eqnarray*}
where the last equivalence follows by \eqref{finalconv1}. Therefore, all the previous inequalities hold as equalities and $\mathcal{F}\left(Du,B_{\tilde{R}}\right)=\mathcal{F}\left(Dv,B_{\tilde{R}}\right)$. The strict convexity of the functional yields that $u=v$ a.e. in $B_{\tilde{R}}$, and since the map $\xi\mapsto V_p(\xi)$ is of class $C^1$, we also have $DV_p\left(Du\right)=DV_p\left(Dv\right)$ almost everywhere in $B_{\tilde{R}}$, and by \eqref{dersecapp1*}, using a standard covering argument, we can conclude with estimate \eqref{mainestimateVp}.
\end{proof}

Thanks to Lemma \ref{differentiabilitylemma}, it is easy to prove the following consequence of Theorem \ref{CGPThm1}.

\begin{corollary}\label{corollary1}
	Let $\Omega\subset\R^n$ be a bounded open set, and $1<p<2$.\\
	Let $u\in W^{1, p}_{\loc}\left(\Omega, \R^N\right)$ be a local minimizer of the functional \eqref{modenergy}, under the assumptions \eqref{F1}--\eqref{F4}, with  
	
	$$f\in L^{\frac{np}{n(p-1)+2-p}}_{\loc}(\Omega)\qquad\text{and}\qquad g\in L^{n}_{\loc}\left(\Omega\right).$$
	
	Then $u\in W^{2, p}_{\loc}\left(\Omega\right)$.
\end{corollary}

\section{A Counterexample}\label{Counterexample}
The aim of this section is to show that we cannot weaken the assumption $f\in L^\frac{np}{n(p-1)+2-p}_{\loc}\left(\Omega\right)$ in the scale of Lebesgue spaces.\\
Our example also shows that this phenomenon is independent of the presence of the coefficients, but it depends only on the sub-quadratic growth of the energy density.\\
For $\alpha\in\R$, let us set

$$\beta=\left(\alpha-1\right)\left(p-1\right)-1,$$ 

and consider the functional 

\begin{equation*}\label{ceEnergy}
	\mathcal{F_\alpha}\left(u,\Omega\right)=\int_{\Omega}\left[\left|Du(x)\right|^p-\alpha\left(n+\beta\right)\left|\alpha\right|^{p-2}\left|x\right|^\beta u(x)\right]dx, 
\end{equation*}

where $\Omega\subset\R^n$ is a bounded open set containing the origin, $1<p<2$, $u:\R^n\to\R$.\\
Using the classical notation for the $p$-Laplacian

$$
\Delta_p u=\div\left(\left|Du\right|^{p-2}\cdot Du\right),
$$

a local minimizer of this functional is a weak solution to the $p$-Poisson equation 

\begin{equation}\label{ceEquation}
	\Delta_p u=f_{\alpha}, 
\end{equation}

with

\begin{equation*}\label{ceDatum}
	f_{\alpha}(x)=\alpha\left(n+\beta\right)\left|\alpha\right|^{p-2}\left|x\right|^\beta. 
\end{equation*}

Before going further, let us notice that \eqref{ceEquation} is an autonomous equation, whose solutions are scalar functions, so the problem we're dealing with is much less general with respect to the assumption we considered in order to prove our main result.\\
It is easy to check that, for any $\alpha\in\R$, the function

\begin{equation*}\label{ceSolution}
u_\alpha(x)=\left|x\right|^\alpha
\end{equation*}

is a solution to \eqref{ceEquation}.\\
Indeed, since, for each $i=1, ..., n$, we have

$$
D_{x_i}u_\alpha(x)=\alpha\left|x\right|^{\alpha-2}x_i,
$$

we get

$$
\left|Du_\alpha(x)\right|=\left|\alpha\right|\left|x\right|^{\alpha-1}.
$$

So, for every $i=1, ..., n$, since $\beta=\left(\alpha-1\right)\left(p-1\right)-1,$ we get

$$
\left|Du_\alpha(x)\right|^{p-2} D_{x_i}u_\alpha(x)=\alpha\left|\alpha\right|^{p-2}\left|x\right|^{(p-1)\cdot(\alpha-1)-1} x_i=\alpha\left|\alpha\right|^{p-2}\left|x\right|^{\beta} x_i
$$

and 

$$
\frac{\partial}{\partial x_i}\left(\left|Du_\alpha(x)\right|^{p-2}D_{x_i}u_\alpha(x)\right)=\alpha\left|\alpha\right|^{p-2}\left|x\right|^{\beta}\left(1+\frac{\beta x_i^2}{\left|x\right|^2}\right), 
$$

so 

\begin{eqnarray*}
	\Delta_p u_\alpha(x)&=&\div\left(\left|Du_\alpha(x)\right|^{p-2}\cdot Du_\alpha(x)\right)=\sum_{i=1}^n\frac{\partial}{\partial x_i}\left(\left|Du_\alpha(x)\right|^{p-2}D_{x_i}u_\alpha(x)\right)\cr\cr
	&=&\alpha\left|\alpha\right|^{p-2}\left|x\right|^{\beta}\sum_{i=1}^n\left(1+\frac{\beta x_i^2}{\left|x\right|^2}\right)=\alpha\left|\alpha\right|^{p-2}\left(n+\beta\right)\left|x\right|^{\beta}\cr\cr
	&=&f_\alpha(x)
\end{eqnarray*}

Moreover, for further needs, we observe that
$$
\left|D^2u_\alpha(x)\right|=c(\alpha)\cdot\left|x\right|^{\alpha-2},
$$
for a constant $c(\alpha)\ge0$.
Choosing 
$$\alpha-1=\frac{2-n}{p}$$ we
have
$$f_\alpha= C_1(\alpha, n, p)|x|^{\frac{(2-n)(p-1)}{p}-1},$$
where $C_1(\alpha, n, p)$ ia real constant, and 
$$
\left|f_\alpha\right|^\frac{np}{n(p-1)+2-p}=C(\alpha, n, p)\left|x\right|^{-n},
$$
with $C(\alpha, n, p)\ge0$.\\
Therefore with such a choice of $\alpha$, we have $f_\alpha\in L^{\frac{np}{n\left(p-1\right)+2-p}-\varepsilon}\left(B_1(0)\right)$ for every $\varepsilon>0$,  but $f_\alpha$ doesn't belong to $L^{\frac{np}{n(p-1)+2-p}}(B_1(0))$.\\
With the same choice of $\alpha$ we have
$$
\left|Du_\alpha(x)\right|^{p-2}\cdot\left|D^2u_\alpha(x)\right|^2=c(n, p)\cdot\left|x\right|^{p\cdot(\alpha-1)-2}=c(n, p)\cdot\left|x\right|^{2-n-2}=c(n, p)\cdot\left|x\right|^{-n}, 
$$
that doesn't belong to $L^1\left(B_1(0)\right)$. Therefore we cannot weaken the assumption on datum $f$ in the scale of Lebesgue spaces and obtain the same regularity for the second derivatives of the solution $u$.

\section{A preliminary higher differentiability result}\label{Preliminarypf}
In this section we consider the following functional
\begin{equation}\label{modenergym}
	\mathcal{F}_m\left(w, \Omega\right)=\int_{\Omega}\left[F\left(x, Dw(x)\right)-f(x)\cdot w(x)+\left(\left|w(x)\right|-a\right)^{2m}_+\right]dx,
\end{equation}
	where $a>0$, $m>1$, and the function $F$ still satisfies \eqref{F1}--\eqref{F4}.\\
It is clear from the definition, and by our assumptions, that the functional in \eqref{modenergym} admits minimizers in $W^{1,p}_{\loc}\left(\Omega\right)\cap L^{2m}_\loc\left(\Omega\right)$.\\
We want to prove the following higher differentiability result for local minimizers of the functional $\mathcal{F}_m$.

\begin{thm}\label{approxmthm}
	Let $\Omega\subset\R^n$ be a bounded open set, $m>1$, $a>0$ and $1<p<2$.\\
	Let $v\in W^{1, p}_{\loc}\left(\Omega, \R^N\right)\cap L^{2m}_{\loc}\left(\Omega, \R^N\right)$ be a local minimizer of the functional \eqref{modenergym}, under the assumptions \eqref{F1}--\eqref{F4}, with  
	
	$$f\in L^{\frac{2m\left(p+2\right)}{2mp+p-2}}_{\loc}\left(\Omega\right)\qquad\mbox{and}\qquad g\in L^{\frac{2m\left(p+2\right)}{2m-p}}_{\loc}\left(\Omega\right).$$
	
	Then $V_p\left(Dv\right)\in W^{1, 2}_{\loc}(\Omega)$, and the estimate
	\begin{eqnarray}\label{approxmestimate}
&&\int_{B_{\frac{R}{2}}}\left|DV_p\left(Dv(x)\right)\right|^2dx\cr\cr
&\le&\frac{c}{R^{\frac{p+2}{p}}}\left[\int_{B_{R}} \left(\mu^2+\left|Dv\left(x\right)\right|^2\right)^\frac{p}{2}dx\right.\cr\cr
&&\left.+\left(\int_{B_{4 R}}\left|v(x)\right|^{2m}dx\right)^\frac{1}{m+1}\cdot\left(\int_{B_{4 R}}\left(\mu^2+\left|Dv(x)\right|^2\right)^\frac{p}{2}dx\right)^\frac{m}{m+1}\right.\cr\cr
&&\left.\int_{B_{R}}g^{\frac{2m\left(p+2\right)}{2m-p}}(x)dx+\int_{B_R}\left|f(x)\right|^\frac{2m\left(p+2\right)}{2mp+p-2}dx+\left|B_R\right|+1\right],
	\end{eqnarray}
	holds true for any ball $B_{4R}\Subset\Omega$.
\end{thm}

For further needs, we notice that
	$$
	\frac{2m\left(p+2\right)}{2mp+p-2}>\frac{m\left(p+2\right)}{mp+m-1}
	$$
	for any $m>1$ as long as $1<p<2$, since it is equivalent to
	$$
	2mp+2m-2>2mp+p-2
	$$
	i.e.
	$$
	2m>p.
	$$
	Let us also notice that
	$$
	\frac{2m\left(p+2\right)}{2mp+p-2}>\frac{p+2}{p}, 
	$$
	and
	$$
	\frac{2m\left(p+2\right)}{2m-p}>p+2
	$$
	for any $m>1$ and $p\in\left(1, 2\right)$.\\
\begin{proof}[Proof of Thorem \ref{approxmthm}]
	{\bf Step 1: the a priori estimate.}\\
	Our first step consists in proving that, if $v\in W^{1, p}_{\loc}\left(\Omega, \R^N\right)\cap L^{2m}_{\loc}\left(\Omega, \R^N\right)$ is a local minimizer of $\mathcal{F}_m$ such that
	$$
	V_p\left(Dv\right)\in W^{1, 2}_{\loc}\left(\Omega\right),
	$$
	estimate \eqref{approxmestimate} holds.
	
	Since $v\in W^{1, p}_{\loc}\left(\Omega, \R^N\right)\cap L^{2m}_{\loc}\left(\Omega, \R^N\right)$ is a local minimizer of $\mathcal{F}_m$, it is a weak solution of the corresponding Euler-Lagrange system, that is, for any $\psi\in C^{\infty}_{0}\left(\Omega,\R^{N}\right)$, we have
	\begin{equation}\label{ELm}
		\int_{\Omega}\langle D_\xi F\left(x,Dv(x)\right),D\psi(x)\rangle dx=\int_{\Omega}\left[f(x)-2m\left(\left|v(x)\right|-a\right)^{2m-1}_+\cdot\frac{v(x)}{\lAbs v(x)\rAbs}\right]\psi(x).
	\end{equation}
Let us fix a ball $B_{4R}\Subset \Omega$ and arbitrary radii $\frac{R}{2}\le r<\tilde{s}<t<\tilde{t}<\lambda r<R,$ with $1<\lambda<2$. Let us consider a cut off function $\eta\in C^\infty_0\left(B_t\right)$ such that $\eta\equiv 1$ on
	$B_{\tilde{s}}$, $\left|D \eta\right|\le \frac{c}{t-\tilde{s}}$ and $\left|D^2 \eta\right|\le \frac{c}{\left(t-\tilde{s}\right)^2}$. From now on, with no	loss of generality, we suppose $R<\frac{1}{4}$.
For $\left|h\right|$ sufficiently small, we can choose, for any $s=1, ..., n$

$$\psi=\tau_{s, -h}\left(\eta^2\tau_{s, h}v\right)$$

as a test function for the equation \eqref{ELm}, and recalling Proposition \ref{findiffpr}, we get
\begin{eqnarray*}
	&&\int_\Omega \left<\tau_{s, h}D_\xi F\left(x, Dv(x)\right), D\left(\eta^2(x)\tau_{s, h}v(x)\right)\right>dx\cr\cr
	&=&\int_\Omega f(x)\cdot\tau_{s, -h}\left(\eta^2(x)\tau_{s, h}v(x)\right)dx\cr\cr
	&&-2m\int_\Omega\tau_{s, h}\left[\left(\left|v(x)\right|-a\right)^{2m-1}_+\cdot\frac{v(x)}{\left|v(x)\right|}\right]\cdot\eta^2(x)\tau_{s, h}v(x)dx, 
\end{eqnarray*}
that is 
\begin{eqnarray*}
	I_1+I_2&:=&\int_\Omega \left<D_\xi F\left(x+he_s, Dv\left(x+he_s\right)\right)-D_\xi F\left(x+he_s, Dv(x)\right), \eta^2(x)\tau_{s, h}Dv(x)\right>dx\cr\cr
	&&+2m\int_\Omega\tau_{s, h}\left[\left(\left|v(x)\right|-a\right)^{2m-1}_+\cdot\frac{v(x)}{\left|v(x)\right|}\right]\cdot\eta^2(x)\tau_{s, h}v(x)dx\cr\cr
	&=&-\int_\Omega \left<D_\xi F\left(x+he_s, Dv(x)\right)-D_\xi F\left(x, Dv(x)\right), \eta^2(x)\tau_{s, h}Dv(x)\right>dx\cr\cr
	&&-2\int_\Omega \left<\tau_{s, h}\left[D_\xi F\left(x, Dv(x)\right)\right], \eta(x)D\eta(x)\otimes\tau_{s, h}v(x)\right>dx\cr\cr
	&&+\int_\Omega f(x)\cdot\tau_{s, -h}\left(\eta^2(x)\tau_{s, h}v(x)\right)dx\cr\cr
	&=:&-I_3-I_4+I_5. 
\end{eqnarray*}

So we have

\begin{equation}\label{fullestimatem*}
	I_1+I_2\le\left|I_3\right|+\left|I_4\right|+\left|I_5\right|.
\end{equation}

By virtue of Lemma \ref{Lemma8}, we have

\begin{equation*}\label{I_2m}
	I_2\ge c(m)\int_\Omega\eta^2(x)\left|\left(\left|v(x+he_s)\right|-a\right)^{m}_+\cdot\frac{v\left(x+he_s\right)}{\left|v\left(x+he_s\right)\right|}-\left(\left|v(x)\right|-a\right)^{m}_+\cdot\frac{v(x)}{\left|v(x)\right|}\right|^2dx\ge0, 
\end{equation*}

so \eqref{fullestimatem*} becomes

\begin{equation}\label{fullestimatem}
	I_1\le\left|I_3\right|+\left|I_4\right|+\left|I_5\right|.
\end{equation}

By assumption \eqref{F2}, we get

\begin{equation}\label{I_1m}
	I_1\ge\nu\int_\Omega\eta^2(x)\left(\mu^2+\left|Dv(x)\right|^2+\left|Dv\left(x+he_s\right)\right|^2\right)^\frac{p-2}{2}\left|\tau_{s, h}Dv(x)\right|^2dx.
\end{equation}

For what concerns the term $I_3$, by \eqref{F4} and using Young's Inequality with exponents $\left(2, 2\right)$, for any $\varepsilon>0$, we have

\begin{eqnarray*}\label{I_3m*}
	\left|I_3\right|&\le&\left|h\right|\int_\Omega\eta^2(x)\left(g(x)+g\left(x+he_s	\right)\right)\left(\mu^2+\left|Dv(x)\right|^2+\left|Dv\left(x+he_s\right)\right|^2\right)^\frac{p-1}{2}\left|\tau_{s, h}Dv(x)\right|dx\cr\cr
	&\le&\varepsilon\int_\Omega\eta^2(x)\left(\mu^2+\left|Dv(x)\right|^2+\left|Dv\left(x+he_s\right)\right|^2\right)^\frac{p-2}{2}\left|\tau_{s, h}Dv(x)\right|^2dx\cr\cr
	&&+c_\varepsilon\left|h\right|^2\int_{\Omega}\eta^2(x)\left(g(x)+g\left(x+he_s\right)\right)^2\left(\mu^2+\left|Dv(x)\right|^2+\left|Dv\left(x+he_s\right)\right|^2\right)^\frac{p}{2}dx.
\end{eqnarray*}

By H\"older's inequality with exponents $\left(\frac{m\left(p+2\right)}{p\left(m+1\right)}, \frac{m\left(p+2\right)}{2m-p}\right)$, the properties of $\eta$ and Lemma \ref{le1}, we get

\begin{eqnarray}\label{I_3m}
	\left|I_3\right|&\le&\varepsilon\int_\Omega\eta^2(x)\left(\mu^2+\left|Dv(x)\right|^2+\left|Dv\left(x+he_s\right)\right|^2\right)^\frac{p-2}{2}\left|\tau_{s, h}Dv(x)\right|^2dx\cr\cr
	&&+c_\varepsilon\left|h\right|^2\left(\int_{B_t}\left(\mu^2+\left|Dv(x)\right|^2+\left|Dv\left(x+he_s\right)\right|^2\right)^\frac{m\left(p+2\right)}{2\left(m+1\right)}dx\right)^\frac{m\left(p+2\right)}{m+1}\cr\cr
	&&\cdot\left(\int_{B_t}\left(g(x)+g\left(x+he_s\right)\right)^{\frac{2m\left(p+2\right)}{2m-p}}dx\right)^\frac{2m-p}{m\left(p+2\right)}\cr\cr
	&\le&\varepsilon\int_\Omega\eta^2(x)\left(\mu^2+\left|Dv(x)\right|^2+\left|Dv\left(x+he_s\right)\right|^2\right)^\frac{p-2}{2}\left|\tau_{s, h}Dv(x)\right|^2dx\cr\cr
	&&+c_\varepsilon\left|h\right|^2\left(\int_{B_{\tilde{t}}}\left(\mu^2+\left|Dv(x)\right|^2\right)^\frac{m\left(p+2\right)}{2\left(m+1\right)}dx\right)^\frac{p\left(m+1\right)}{m\left(p+2\right)}\cdot\left(\int_{B_{\lambda r}}g^{\frac{2m\left(p+2\right)}{2m-p}}(x)dx\right)^\frac{2m-p}{m\left(p+2\right)}.
\end{eqnarray}

Let us consider, now, the term $I_4$.
We have 

\begin{eqnarray*}
	I_4&=&2\int_\Omega \left<\tau_{s, h}\left[D_\xi F\left(x, Dv(x)\right)\right], \eta(x)D\eta(x)\otimes\tau_{s, h}v(x)\right>dx\cr\cr
	&=&2\int_\Omega \left<D_\xi F\left(x, Dv(x)\right), \tau_{s, -h}\left[\eta(x)D\eta(x)\otimes\tau_{s, h}v(x)\right]\right>dx,
\end{eqnarray*}

so, by \eqref{F1}, we get

\begin{eqnarray*}\label{I_4m*}
	\left|I_4\right|&\le&c\int_\Omega \left(\mu^2+\left|Dv(x)\right|^2\right)^\frac{p-1}{2}\left|\tau_{s, -h}\left[\eta(x)D\eta(x)\otimes\tau_{s, h}v(x)\right]\right|dx.
\end{eqnarray*}

We can treat this term as we did after \eqref{I_3*} in the proof of Theorem \ref{CGPThm1}, using \eqref{tau1} with $v$ in place of $u$, thus getting

\begin{equation}\label{I_4m}
	\left|I_4\right|\le\sigma\int_{B_{\tilde{t}}}\left|\tau_{s, h}V_p\left(Dv(x)\right)\right|^2dx+\frac{c_\sigma\left|h\right|^2}{\left(t-\tilde{s}\right)^2}\int_{B_{R}} \left(\mu^2+\left|Dv\left(x\right)\right|^2\right)^\frac{p}{2}dx,
\end{equation}
for any $\sigma>0$.\\

In order to estimate the term $I_5$, arguing as we did in \eqref{I_4'}, we have

\begin{eqnarray*}
	I_5&=&\int_\Omega\eta^2\left(x\right)f(x)\tau_{s, -h}\left(\tau_{s, h}v(x)\right)dx\cr\cr
	&&+\int_\Omega\left[\eta\left(x-he_s\right)+\eta(x)\right]f(x)\tau_{s, -h}\eta(x)\tau_{s, h}v\left(x-he_s\right)dx\cr\cr
	&=:&J_1+J_2,
\end{eqnarray*}

which implies 

\begin{equation}\label{I_5m*}
	\left|I_5\right|\le\left|J_1\right|+\left|J_2\right|
\end{equation}

Let us consider the term $J_1$. By virtue of the properties of $\eta$ and using H\"older's inequality with exponents $\left(\frac{2m\left(p+2\right)}{2mp+p-2}, \frac{2m\left(p+2\right)}{4m+2-p}\right)$, we have

\begin{eqnarray}\label{J_1m*}
	\left|J_1\right|&\le&\int_{B_t}\left|f(x)\right|\left|\tau_{s, -h}\left(\tau_{s, h}v(x)\right)\right|dx\cr\cr
	&\le&\left(\int_{B_t}\left|f(x)\right|^\frac{2m\left(p+2\right)}{2mp+p-2}dx\right)^\frac{2mp+p-2}{2m\left(p+2\right)}\cdot\left(\int_{B_t}\left|\tau_{s, -h}\left(\tau_{s, h}v(x)\right)\right|^\frac{2m\left(p+2\right)}{4m+2-p}dx\right)^\frac{4m+2-p}{2m\left(p+2\right)}\cr\cr
	&\le&\left|h\right|\left(\int_{B_t}\left|f(x)\right|^\frac{2m\left(p+2\right)}{2mp+p-2}dx\right)^\frac{2mp+p-2}{2m\left(p+2\right)}\cdot\left(\int_{B_{\tilde{t}}}\left|\tau_{s, h}Dv(x)\right|^\frac{2m\left(p+2\right)}{4m+2-p}dx\right)^\frac{4m+2-p}{2m\left(p+2\right)}, 
\end{eqnarray}

where, in the last line we applied Lemma \ref{le1} since, by the a priori assumption $V_p\left(Dv\right)\in W^{1,2}_\loc\left(\Omega\right)$ and Remark \ref{rmk3}, we have $Dv\in L^{\frac{m\left(p+2\right)}{m+1}}_\loc\left(\Omega\right)$, which implies $Dv\in L^{\frac{2m\left(p+2\right)}{4m+2-p}}_\loc\left(\Omega\right)$ since, for any $m>1$ and $1<p<2$, we have $\frac{2m\left(p+2\right)}{4m+2-p}<\frac{m\left(p+2\right)}{m+1}$.\\
Let us consider the second integral in \eqref{J_1m*}. By virtue of \eqref{lemma6GPestimate1}, and using H\"older's inequality with exponents $\left(\frac{4m+2-p}{m\left(p+2\right)}, \frac{4m+2-p}{\left(2-p\right)\left(m+1\right)}\right)$, we have

\begin{eqnarray}\label{tauVpm}
	\int_{B_{\tilde{t}}}\left|\tau_{s, h}Dv(x)\right|^\frac{2m\left(p+2\right)}{4m+2-p}dx&\le&\int_{B_{\tilde{t}}}\left(\mu^2+\left|Dv(x)\right|^2+\left|Dv\left(x+he_s\right)\right|^2\right)^{\frac{2-p}{4}\cdot\frac{2m\left(p+2\right)}{4m+2-p}}\cr\cr
	&&\cdot\left|\tau_{s, h}V_p\left(Dv(x)\right)\right|^{\frac{2m\left(p+2\right)}{4m+2-p}}dx\cr\cr
	&\le&\left(\int_{B_{\tilde{t}}}\left(\mu^2+\left|Dv(x)\right|^2+\left|Dv\left(x+he_s\right)\right|^2\right)^{\frac{m\left(p+2\right)}{2\left(m+1\right)}}dx\right)^\frac{\left(2-p\right)\left(m+1\right)}{4m+2-p}\cr\cr
	&&\cdot\left(\int_{B_{\tilde{t}}}\left|\tau_{s, h}V_p\left(Dv(x)\right)\right|^{2}dx\right)^\frac{m\left(p+2\right)}{4m+2-p}.
\end{eqnarray}

Inserting \eqref{tauVpm} into \eqref{J_1m*}, we get

\begin{eqnarray}\label{J_1m**}
	\left|J_1\right|&\le&\left|h\right|\left(\int_{B_t}\left|f(x)\right|^\frac{2m\left(p+2\right)}{2mp+p-2}dx\right)^\frac{2mp+p-2}{2m\left(p+2\right)}\cr\cr
	&&\cdot\left(\int_{B_{\tilde{t}}}\left(\mu^2+\left|Dv(x)\right|^2+\left|Dv\left(x+he_s\right)\right|^2\right)^{\frac{m\left(p+2\right)}{2\left(m+1\right)}}dx\right)^\frac{\left(2-p\right)\left(m+1\right)}{2m\left(p+2\right)}\cr\cr
	&&\cdot\left(\int_{B_{\tilde{t}}}\left|\tau_{s, h}V_p\left(Dv(x)\right)\right|^{2}dx\right)^\frac{1}{2}, 
\end{eqnarray}

Using Lemma \ref{le1} and Young's Inequality with exponents $\left(\frac{2m\left(p+2\right)}{2mp+p-2}, \frac{2m\left(p+2\right)}{\left(2-p\right)\left(m+1\right)}, 2\right)$, we get 

\begin{eqnarray}\label{J_1m}
	\left|J_1\right|&\le&c_\sigma\left|h\right|^2\int_{B_t}\left|f(x)\right|^\frac{2m\left(p+2\right)}{2mp+p-2}dx+\sigma\left|h\right|^2\int_{B_{\lambda r}}\left(\mu^2+\left|Dv(x)\right|^2\right)^{\frac{m\left(p+2\right)}{2\left(m+1\right)}}dx\cr\cr
	&&+\sigma\int_{B_{\tilde{t}}}\left|\tau_{s, h}V_p\left(Dv(x)\right)\right|^{2}dx,
\end{eqnarray}

for any $\sigma>0$.\\
For what concerns the term $J_2$, by the properties of $\eta$, as in \eqref{J_2*}, we have

\begin{eqnarray*}\label{J_2m*}
	\left|J_2\right|&\le&\int_{B_t} \left|f(x)\right|\left|\tau_{s, -h}\eta(x)\right|\left|\tau_{s, h}v\left(x-he_s\right)\right|dx\cr\cr
	&\le&\frac{c\left|h\right|}{t-\tilde{s}}\int_{B_t} \left|f(x)\right|\left|\tau_{s, h}v\left(x-he_s\right)\right|dx.
\end{eqnarray*}

Now, if we apply H\"older's Inequality with exponents $\left(\frac{m\left(p+2\right)}{mp+m-1}, \frac{m\left(p+2\right)}{m+1}\right)$, we get

\begin{eqnarray}\label{J_2m}
	\left|J_2\right|&\le&\frac{c\left|h\right|}{t-\tilde{s}}\left(\int_{B_t} \left|f(x)\right|^\frac{m\left(p+2\right)}{mp+m-1}dx\right)^\frac{mp+m-1}{m\left(p+2\right)}\cr\cr
	&&\cdot\left(\int_{B_t}\left|\tau_{s, h}v\left(x-he_s\right)\right|^\frac{m\left(p+2\right)}{m+1}dx\right)^\frac{m+1}{m\left(p+2\right)}\cr\cr
	&\le&\frac{c\left|h\right|^2}{t-\tilde{s}}\left(\int_{B_t} \left|f(x)\right|^\frac{m\left(p+2\right)}{mp+m-1}dx\right)^\frac{mp+m-1}{m\left(p+2\right)}\cr\cr
	&&\cdot\left(\int_{B_{\lambda r}}\left|Dv\left(x\right)\right|^\frac{m\left(p+2\right)}{m+1}dx\right)^\frac{m+1}{m\left(p+2\right)},
\end{eqnarray} 
where we also used Lemma \ref{le1}, since $Dv\in L^{\frac{m\left(p+2\right)}{m+1}}_{\loc}\left(\Omega\right)$.\\
By virtue of \eqref{J_1m} and \eqref{J_2m}, \eqref{I_5m*} gives

\begin{eqnarray}\label{I_5m}
	\left|I_5\right|&\le&c_\sigma\left|h\right|^2\int_{B_t}\left|f(x)\right|^\frac{2m\left(p+2\right)}{2mp+p-2}dx+\sigma\left|h\right|^2\int_{B_{\lambda r}}\left(\mu^2+\left|Dv(x)\right|^2\right)^{\frac{m\left(p+2\right)}{2\left(m+1\right)}}dx\cr\cr
	&&+\sigma\int_{B_{\tilde{t}}}\left|\tau_{s, h}V_p\left(Dv(x)\right)\right|^{2}dx\cr\cr
	&&+\frac{c\left|h\right|^2}{t-\tilde{s}}\left(\int_{B_t} \left|f(x)\right|^\frac{m\left(p+2\right)}{mp+m-1}dx\right)^\frac{mp+m-1}{m\left(p+2\right)}\cr\cr
	&&\cdot\left(\int_{B_{\lambda r}}\left|Dv\left(x\right)\right|^\frac{m\left(p+2\right)}{m+1}dx\right)^\frac{m+1}{m\left(p+2\right)}\cr\cr
	&\le&2\sigma\left|h\right|^2\int_{B_{\lambda r}}\left(\mu^2+\left|Dv(x)\right|^2\right)^{\frac{m\left(p+2\right)}{2\left(m+1\right)}}dx+\sigma\int_{B_{\tilde{t}}}\left|\tau_{s, h}V_p\left(Dv(x)\right)\right|^{2}dx\cr\cr
	&&+c_\sigma\left|h\right|^2\int_{B_t}\left|f(x)\right|^\frac{2m\left(p+2\right)}{2mp+p-2}dx+\frac{c_\sigma\left|h\right|^2}{\left(t-\tilde{s}\right)^\frac{m\left(p+2\right)}{mp+m-1}}\int_{B_t} \left|f(x)\right|^\frac{m\left(p+2\right)}{mp+m-1}dx, 
\end{eqnarray}

where we also used Young's Inequality with exponents $\left(\frac{m\left(p+2\right)}{mp+m-1} ,\frac{m\left(p+2\right)}{m+1}\right)$.\\
Plugging \eqref{I_1m}, \eqref{I_3m}, \eqref{I_4m} and \eqref{I_5m} into \eqref{fullestimatem}, and choosing $\varepsilon<\frac{\nu}{2}$, we get 

\begin{eqnarray*}\label{fullestimatem2}
	&&\int_{\Omega}\eta^2(x)\left|\tau_{s, h}Dv\left(x\right)\right|^2\left(\mu^2+\left|Dv\left(x+he_s\right)\right|^2+\left|Dv\left(x\right)\right|^2\right)^\frac{p-2}{2}dx\cr\cr
	&\le&c\left|h\right|^2\left(\int_{B_{\tilde{t}}}\left(\mu^2+\left|Dv(x)\right|^2\right)^\frac{m\left(p+2\right)}{2\left(m+1\right)}dx\right)^\frac{p\left(m+1\right)}{m\left(p+2\right)}\cdot\left(\int_{B_{\lambda r}}g^{\frac{2m\left(p+2\right)}{2m-p}}(x)dx\right)^\frac{2m-p}{m\left(p+2\right)}\cr\cr
	&&+2\sigma\int_{B_{\tilde{t}}}\left|\tau_{s, h}V_p\left(Dv(x)\right)\right|^2dx+\frac{c_\sigma\left|h\right|^2}{\left(t-\tilde{s}\right)^2}\int_{B_{R}} \left(\mu^2+\left|Dv\left(x\right)\right|^2\right)^\frac{p}{2}dx\cr\cr
	&&+c_\sigma\left|h\right|^2\int_{B_t}\left|f(x)\right|^\frac{2m\left(p+2\right)}{2mp+p-2}dx+\frac{c_\sigma\left|h\right|^2}{\left(t-\tilde{s}\right)^\frac{m\left(p+2\right)}{mp+m-1}}\int_{B_t} \left|f(x)\right|^\frac{m\left(p+2\right)}{mp+m-1}dx\cr\cr
	&&+2\sigma\left|h\right|^2\int_{B_{\lambda r}}\left(\mu^2+\left|Dv(x)\right|^2\right)^{\frac{m\left(p+2\right)}{2\left(m+1\right)}}dx
\end{eqnarray*}

which, by virtue of Lemma \ref{lemma6GP}, and using Young's inequality with exponents $\left(\frac{m\left(p+2\right)}{p\left(m+1\right)}, \frac{m\left(p+2\right)}{2m-p}\right)$ implies

\begin{eqnarray}\label{tauV_pm*}
	&&\int_{\Omega}\eta^2(x)\left|\tau_{s, h}Dv\left(x\right)\right|^2\left(\mu^2+\left|Dv\left(x+he_s\right)\right|^2+\left|Dv\left(x\right)\right|^2\right)^\frac{p-2}{2}dx\cr\cr
	&\le&2\sigma\int_{B_{\tilde{t}}}\left|\tau_{s, h}V_p\left(Dv(x)\right)\right|^2dx+3\sigma\left|h\right|^2\int_{B_{\lambda r}}\left(\mu^2+\left|Dv(x)\right|^2\right)^{\frac{m\left(p+2\right)}{2\left(m+1\right)}}dx\cr\cr
	&&+\frac{c_\sigma\left|h\right|^2}{\left(t-\tilde{s}\right)^2}\int_{B_{R}} \left(\mu^2+\left|Dv\left(x\right)\right|^2\right)^\frac{p}{2}dx+c_\sigma\left|h\right|^2\int_{B_{\lambda r}}g^{\frac{2m\left(p+2\right)}{2m-p}}(x)dx\cr\cr
	&&+c_\sigma\left|h\right|^2\int_{B_t}\left|f(x)\right|^\frac{2m\left(p+2\right)}{2mp+p-2}dx+\frac{c_\sigma\left|h\right|^2}{\left(t-\tilde{s}\right)^\frac{m\left(p+2\right)}{mp+m-1}}\int_{B_t} \left|f(x)\right|^\frac{m\left(p+2\right)}{mp+m-1}dx.
\end{eqnarray}
Applying Lemma \ref{le1}, \eqref{tauV_pm*} becomes

\begin{eqnarray}\label{tauV_pm**}
	&&\int_{\Omega}\eta^2(x)\left|\tau_{s, h}V_p\left(Dv(x)\right)\right|^2dx\cr\cr
	&\le&3\sigma\left|h\right|^2\int_{B_{\lambda r}}\left(\mu^2+\left|Dv(x)\right|^2\right)^\frac{m\left(p+2\right)}{2\left(m+1\right)}dx+c\cdot\sigma\left|h\right|^2\int_{B_{\lambda r}}\left|DV_p\left(Dv(x)\right)\right|^2dx\cr\cr
	&&+\frac{c_\sigma\left|h\right|^2}{\left(t-\tilde{s}\right)^2}\int_{B_{R}} \left(\mu^2+\left|Dv\left(x\right)\right|^2\right)^\frac{p}{2}dx+c_\sigma\left|h\right|^2\int_{B_{\lambda r}}g^{\frac{2m\left(p+2\right)}{2m-p}}(x)dx\cr\cr
	&&+c_\sigma\left|h\right|^2\int_{B_t}\left|f(x)\right|^\frac{2m\left(p+2\right)}{2mp+p-2}dx+\frac{c_\sigma\left|h\right|^2}{\left(t-\tilde{s}\right)^\frac{m\left(p+2\right)}{mp+m-1}}\int_{B_t} \left|f(x)\right|^\frac{m\left(p+2\right)}{mp+m-1}dx.
\end{eqnarray}
Let us observe that, for any $m>1$ and $1<p<2$, we have
$$
\frac{m\left(p+2\right)}{mp+m-1}\le\frac{p+2}{p}, 
$$
hence
$$
\max\Set{2, \frac{m\left(p+2\right)}{mp+m-1}}\le\max\Set{2, \frac{p+2}{p}}=\frac{p+2}{p}. 
$$
Hence, since $t-\tilde{s}<1$, by \eqref{tauV_pm**} we deduce
\begin{eqnarray}\label{tauV_pm'}
	&&\int_{\Omega}\eta^2(x)\left|\tau_{s, h}V_p\left(Dv(x)\right)\right|^2dx\cr\cr
	&\le&3\sigma\left|h\right|^2\int_{B_{\lambda r}}\left(\mu^2+\left|Dv(x)\right|^2\right)^\frac{m\left(p+2\right)}{2\left(m+1\right)}dx+c\cdot\sigma\left|h\right|^2\int_{B_{\lambda r}}\left|DV_p\left(Dv(x)\right)\right|^2dx\cr\cr
	&&+\frac{c_\sigma\left|h\right|^2}{\left(t-\tilde{s}\right)^{\frac{p+2}{p}}}\left[\int_{B_{R}}g^{\frac{2m\left(p+2\right)}{2m-p}}(x)dx+\int_{B_{R}} \left(\mu^2+\left|Dv\left(x\right)\right|^2\right)^\frac{p}{2}dx\right.\cr\cr
	&&\left.+\int_{B_R}\left|f(x)\right|^\frac{2m\left(p+2\right)}{2mp+p-2}dx+\int_{B_R} \left|f(x)\right|^\frac{m\left(p+2\right)}{mp+m-1}dx\right].
\end{eqnarray}

Let us notice that, since $\frac{m\left(p+2\right)}{mp+m-1}<\frac{2m\left(p+2\right)}{2mp+p-2}$, we have $L^{\frac{2m\left(p+2\right)}{2mp+p-2}}_\loc\left(\Omega\right)\hookrightarrow L^{\frac{m\left(p+2\right)}{mp+m-1}}_\loc\left(\Omega\right)$, and using Young's inequality with exponents  $\left(\frac{2\left(mp+m-1\right)}{2mp+p-2},  \frac{2\left(mp+m-1\right)}{2m-p}\right)$, we have

\begin{equation}\label{Immersion}
	\int_{B_{R}} \left|f(x)\right|^\frac{m\left(p+2\right)}{mp+m-1}dx\le c\left|B_R\right|+c\int_{B_{R}} \left|f(x)\right|^{\frac{2m\left(p+2\right)}{2mp+p-2}}dx.
\end{equation}

So,  plugging \eqref{Immersion} into \eqref{tauV_pm'}, we get

\begin{eqnarray}\label{tauV_pm}
	&&\int_{\Omega}\eta^2(x)\left|\tau_{s, h}V_p\left(Dv(x)\right)\right|^2dx\cr\cr
	&\le&3\sigma\left|h\right|^2\int_{B_{\lambda r}}\left(\mu^2+\left|Dv(x)\right|^2\right)^\frac{m\left(p+2\right)}{2\left(m+1\right)}dx+c\cdot\sigma\left|h\right|^2\int_{B_{\lambda r}}\left|DV_p\left(Dv(x)\right)\right|^2dx\cr\cr
	&&+\frac{c_\sigma\left|h\right|^2}{\left(t-\tilde{s}\right)^{\frac{p+2}{p}}}\left[\int_{B_{R}}g^{\frac{2m\left(p+2\right)}{2m-p}}(x)dx+\int_{B_{R}} \left(\mu^2+\left|Dv\left(x\right)\right|^2\right)^\frac{p}{2}dx\right.\cr\cr
	&&\left.+\int_{B_R}\left|f(x)\right|^\frac{2m\left(p+2\right)}{2mp+p-2}dx+\left|B_R\right|\right].
\end{eqnarray}

Since, by our a priori assumption, $V_p\left(Dv\right)\in W^{1, 2}_\loc\left(\Omega\right)$, and \eqref{tauV_pm} holds for any $s=1, ..., n$, Lemma \ref{Giusti8.2} implies

\begin{eqnarray*}
	&&\int_{\Omega}\eta^2(x)\left|DV_p\left(Dv(x)\right)\right|^2dx\cr\cr
	&\le&3\sigma\int_{B_{\lambda r}}\left(\mu^2+\left|Dv(x)\right|^2\right)^\frac{m\left(p+2\right)}{2\left(m+1\right)}dx+c\cdot\sigma\int_{B_{\lambda r}}\left|DV_p\left(Dv(x)\right)\right|^2dx\cr\cr
	&&+\frac{c_\sigma}{\left(t-\tilde{s}\right)^{\frac{p+2}{p}}}\left[\int_{B_{R}}g^{\frac{2m\left(p+2\right)}{2m-p}}(x)dx+\int_{B_{R}} \left(\mu^2+\left|Dv\left(x\right)\right|^2\right)^\frac{p}{2}dx\right.\cr\cr
	&&\left.+\int_{B_R}\left|f(x)\right|^\frac{2m\left(p+2\right)}{2mp+p-2}dx+\left|B_R\right|\right], 
\end{eqnarray*}

and by the properties of $\eta$, we get

\begin{eqnarray}\label{DV_pm**}
	&&\int_{B_{\tilde{s}}}\left|DV_p\left(Dv(x)\right)\right|^2dx\cr\cr
	&\le&c\cdot\sigma\int_{B_{\lambda r}}\left|DV_p\left(Dv(x)\right)\right|^2dx\cr\cr
	&&+\frac{c_\sigma}{\left(t-\tilde{s}\right)^{\frac{p+2}{p}}}\left[\int_{B_{R}}g^{\frac{2m\left(p+2\right)}{2m-p}}(x)dx+\int_{B_{R}} \left(\mu^2+\left|Dv\left(x\right)\right|^2\right)^\frac{p}{2}dx\right.\cr\cr
	&&\left.+\int_{B_R}\left|f(x)\right|^\frac{2m\left(p+2\right)}{2mp+p-2}dx+\left|B_R\right|\right]\cr\cr
	&&+3\sigma\int_{B_{\lambda r}}\left(\mu^2+\left|Dv(x)\right|^2\right)^\frac{m\left(p+2\right)}{2\left(m+1\right)}dx, 
\end{eqnarray}

Let us remind that, since we choosed $B_{4R}\Subset\Omega$, $\frac{R}{2}<r<\tilde{s}<t<\tilde{t}<\lambda r<R$, with $1<\lambda<2$ and $R<\frac{1}{4}$, we also have $\lambda r<\lambda\tilde{s}<\lambda t<\lambda^2 r<4r<4R<1$.\\
Choosing a cut-off function $\phi\in C^\infty_{0}\left(B_{\lambda t}\right)$ such that $0\le\phi\le1$, $\phi\equiv 1$ on $B_{\lambda \tilde{s}}$ and $\left|D\phi\right|\le\frac{c}{\lambda\left(t-\tilde{s}\right)}$, we have

\begin{equation*}\label{Iterationm2}
	\int_{B_{\lambda r}}\left(\mu^2+\left|Dv(x)\right|^2\right)^\frac{m\left(p+2\right)}{2\left(m+1\right)}dx\le\left|B_R\right|+\int_{B_{\lambda t}}\phi^{\frac{m}{m+1}\left(p+2\right)}(x)\left|Dv(x)\right|^{\frac{m}{m+1}\left(p+2\right)}dx.
\end{equation*} 

where we also used that $\mu\in[0, 1]$ and $\lambda r<R$.
Therefore, applying \eqref{2.1GP}, we get 

\begin{eqnarray}\label{Iterationm2*}
	&&\int_{B_{\lambda r}}\left(\mu^2+\left|Dv(x)\right|^2\right)^\frac{m\left(p+2\right)}{2\left(m+1\right)}dx\cr\cr
	&\le&(p+2)^2\left(\int_{B_{\lambda t}}\phi^{\frac{m}{m+1}(p+2)}\left|v(x)\right|^{2m}dx\right)^\frac{1}{m+1}\cr\cr
	&&\cdot\left[\left(\int_{B_{\lambda t}}\phi^{\frac{m}{m+1}(p+2)}\left|D\phi\right|^2\left(\mu^2+\left|Dv(x)\right|^2\right)^\frac{p}{2}dx\right)^\frac{m}{m+1}\right.\cr\cr
	&&\left.+n\left(\int_{B_{\lambda t}}\phi^{\frac{m}{m+1}(p+2)}\left(\mu^2+\left|Dv(x)\right|^2\right)^\frac{p-2}{2}\left|D^2v(x)\right|^2dx\right)^\frac{m}{m+1}\right]+\left|B_R\right|\cr\cr
	&\le&c\left(n, p\right)\left(\int_{B_{4R}}\left|v(x)\right|^{2m}dx\right)^\frac{1}{m+1}\cr\cr
	&&\cdot\left[\frac{1}{\lambda^2\left(t-\tilde{s}\right)^2}\left(\int_{B_{4R}}\left(\mu^2+\left|Dv(x)\right|^2\right)^\frac{p}{2}dx\right)^\frac{m}{m+1}\right.\cr\cr
	&&\left.+\left(\int_{B_{\lambda^2 r}}\left|DV_p\left(Dv(x)\right)\right|^2dx\right)^\frac{m}{m+1}\right]+\left|B_R\right|, 
\end{eqnarray} 
where we used the properties of $\phi$, \eqref{lemma6GPestimate2}, and the fact that $\lambda t<\lambda^2 r<4R$.\\
The elementary inequality
$$
b^\frac{m}{m+1}\le b+1,\qquad\mbox{ for any }m>1\mbox{ and }b\ge0,
$$
implies

\begin{equation}\label{DVp}
	\left(\int_{B_{\lambda^2 r}}\left|DV_p\left(Dv(x)\right)\right|^2dx\right)^\frac{m}{m+1}\le\int_{B_{\lambda^2 r}}\left|DV_p\left(Dv(x)\right)\right|^2dx+1.
\end{equation}

Now, if we recall that $1<\lambda<2$, $t-\tilde{s}<\lambda\left( t-\tilde{s}\right)<1$ and $\frac{p+2}{p}\ge2$, \eqref{Iterationm2*} and \eqref{DVp} imply

\begin{eqnarray*}
	&&\int_{B_{\tilde{s}}}\left|DV_p\left(Dv(x)\right)\right|^2dx\cr\cr
	&\le&c\cdot\sigma\left(\int_{B_{4 R}}\left|v(x)\right|^{2m}dx\right)^\frac{1}{m+1}\cdot\int_{B_{\lambda^2 r}}\left|DV_p\left(Dv(x)\right)\right|^2dx\cr\cr
	&&+\frac{\lambda^2+1}{\lambda^2}\cdot\frac{c_\sigma}{\left(t-\tilde{s}\right)^{\frac{p+2}{p}}}\left[\int_{B_{R}}g^{\frac{2m\left(p+2\right)}{2m-p}}(x)dx+\int_{B_{R}} \left(\mu^2+\left|Dv\left(x\right)\right|^2\right)^\frac{p}{2}dx\right.\cr\cr
	&&\left.+\left(\int_{B_{4 R}}\left|v(x)\right|^{2m}dx\right)^\frac{1}{m+1}\cdot\left(\int_{B_{4 R}}\left(\mu^2+\left|Dv(x)\right|^2\right)^\frac{p}{2}dx\right)^\frac{m}{m+1}\right.\cr\cr
	&&\left.+\int_{B_R}\left|f(x)\right|^\frac{2m\left(p+2\right)}{2mp+p-2}dx+\left|B_R\right|+1\right]\cr\cr
	&\le&c\cdot\sigma\left(\int_{B_{4 R}}\left|v(x)\right|^{2m}dx\right)^\frac{1}{m+1}\cdot\int_{B_{\lambda^2 r}}\left|DV_p\left(Dv(x)\right)\right|^2dx\cr\cr
	&&+\frac{c_\sigma}{\left(t-\tilde{s}\right)^{\frac{p+2}{p}}}\left[\int_{B_{R}}g^{\frac{2m\left(p+2\right)}{2m-p}}(x)dx+\int_{B_{R}} \left(\mu^2+\left|Dv\left(x\right)\right|^2\right)^\frac{p}{2}dx\right.\cr\cr
	&&\left.+\left(\int_{B_{4 R}}\left|v(x)\right|^{2m}dx\right)^\frac{1}{m+1}\cdot\left(\int_{B_{4 R}}\left(\mu^2+\left|Dv(x)\right|^2\right)^\frac{p}{2}dx\right)^\frac{m}{m+1}\right.\cr\cr
	&&\left.+\int_{B_R}\left|f(x)\right|^\frac{2m\left(p+2\right)}{2mp+p-2}dx+\left|B_R\right|+1\right],
\end{eqnarray*}

which, if we choose $\sigma>0$ such that

$$
c\cdot\sigma\left(\int_{B_{4 R}}\left|v(x)\right|^{2m}dx\right)^\frac{1}{m+1}<\frac{1}{2},
$$

becomes

\begin{eqnarray}\label{DV_pmb}
	\int_{B_{\tilde{s}}}\left|DV_p\left(Dv(x)\right)\right|^2dx&\le&\frac{1}{2}\int_{B_{\lambda^2 r}}\left|DV_p\left(Dv(x)\right)\right|^2dx\cr\cr
	&&+\frac{c}{\left(t-\tilde{s}\right)^{\frac{p+2}{p}}}\left[\int_{B_{R}}g^{\frac{2m\left(p+2\right)}{2m-p}}(x)dx+\int_{B_{R}} \left(\mu^2+\left|Dv\left(x\right)\right|^2\right)^\frac{p}{2}dx\right.\cr\cr
	&&\left.+\left(\int_{B_{4 R}}\left|v(x)\right|^{2m}dx\right)^\frac{1}{m+1}\cdot\left(\int_{B_{4 R}}\left(\mu^2+\left|Dv(x)\right|^2\right)^\frac{p}{2}dx\right)^\frac{m}{m+1}\right.\cr\cr
	&&\left.+\int_{B_R}\left|f(x)\right|^\frac{2m\left(p+2\right)}{2mp+p-2}dx+\left|B_R\right|+1\right], 
\end{eqnarray}

Since \eqref{DV_pmb} holds for any $\frac{R}{2}<r<\tilde{s}<t<\tilde{t}<\lambda r<R,$ with  $1<\lambda<2$, with a constant $c$ depending on $n, p, L, \nu, L_1, \ell$, but is independent of the radii, passing to the limit as $\tilde{s}\to r$ and $t\to\lambda r$, we get

\begin{eqnarray*}
\int_{B_{r}}\left|DV_p\left(Dv(x)\right)\right|^2dx&\le&\frac{1}{2}\int_{B_{\lambda^2 r}}\left|DV_p\left(Dv(x)\right)\right|^2dx\cr\cr
&&+\frac{c}{r^{\frac{p+2}{p}}\left(\lambda-1\right)^{\frac{p+2}{p}}}\left[\int_{B_{R}} \left(\mu^2+\left|Dv\left(x\right)\right|^2\right)^\frac{p}{2}dx\right.\cr\cr
&&\left.+\left(\int_{B_{4 R}}\left|v(x)\right|^{2m}dx\right)^\frac{1}{m+1}\cdot\left(\int_{B_{4 R}}\left(\mu^2+\left|Dv(x)\right|^2\right)^\frac{p}{2}dx\right)^\frac{m}{m+1}\right.\cr\cr
&&+\left.\int_{B_{R}}g^{\frac{2m\left(p+2\right)}{2m-p}}(x)dx+\int_{B_R}\left|f(x)\right|^\frac{2m\left(p+2\right)}{2mp+p-2}dx+\left|B_R\right|+1\right],  
\end{eqnarray*}

and since $1<\lambda<2$, we have

\begin{eqnarray}\label{DV_pmIteration1}
	\int_{B_{r}}\left|DV_p\left(Dv(x)\right)\right|^2dx&\le&\frac{1}{2}\int_{B_{\lambda^2 r}}\left|DV_p\left(Dv(x)\right)\right|^2dx\cr\cr
	&&+\frac{c}{r^{\frac{p+2}{p}}\left(\lambda^2-1\right)^{\frac{p+2}{p}}}\left[\int_{B_{R}} \left(\mu^2+\left|Dv\left(x\right)\right|^2\right)^\frac{p}{2}dx\right.\cr\cr
	&&\left.+\left(\int_{B_{4 R}}\left|v(x)\right|^{2m}dx\right)^\frac{1}{m+1}\cdot\left(\int_{B_{4 R}}\left(\mu^2+\left|Dv(x)\right|^2\right)^\frac{p}{2}dx\right)^\frac{m}{m+1}\right.\cr\cr
	&&+\left.\int_{B_{R}}g^{\frac{2m\left(p+2\right)}{2m-p}}(x)dx+\int_{B_R}\left|f(x)\right|^\frac{2m\left(p+2\right)}{2mp+p-2}dx+\left|B_R\right|+1\right].
\end{eqnarray}

Now, if we set 
$$
h(r)=\int_{B_{r}}\left|DV_p\left(Dv(x)\right)\right|^2dx,
$$
\begin{eqnarray*}
A&=&\left[\int_{B_{R}} \left(\mu^2+\left|Dv\left(x\right)\right|^2\right)^\frac{p}{2}dx\right.\cr\cr
&&\left.+\left(\int_{B_{4 R}}\left|v(x)\right|^{2m}dx\right)^\frac{1}{m+1}\cdot\left(\int_{B_{4 R}}\left(\mu^2+\left|Dv(x)\right|^2\right)^\frac{p}{2}dx\right)^\frac{m}{m+1}\right.\cr\cr
&&+\left.\int_{B_{R}}g^{\frac{2m\left(p+2\right)}{2m-p}}(x)dx+\int_{B_R}\left|f(x)\right|^\frac{2m\left(p+2\right)}{2mp+p-2}dx+\left|B_R\right|+1\right], 
\end{eqnarray*}
and 
$$
B=0
$$
since \eqref{DV_pmIteration1} holds for any $\lambda\in(1, 2)$, we can apply Lemma \ref{iter} with
$$
\theta=\frac{1}{2}\qquad\mbox{ and }\qquad\gamma=\frac{p+2}{p}, 
$$
thus getting
\begin{eqnarray}\label{apriorimPf}
	\int_{B_{\frac{R}{2}}}\left|DV_p\left(Dv(x)\right)\right|^2dx&\le&\frac{c}{R^{\frac{p+2}{p}}}\left[\int_{B_{R}} \left(\mu^2+\left|Dv\left(x\right)\right|^2\right)^\frac{p}{2}dx\right.\cr\cr
	&&\left.+\left(\int_{B_{4 R}}\left|v(x)\right|^{2m}dx\right)^\frac{1}{m+1}\cdot\left(\int_{B_{4 R}}\left(\mu^2+\left|Dv(x)\right|^2\right)^\frac{p}{2}dx\right)^\frac{m}{m+1}\right.\cr\cr
	&&+\left.\int_{B_{R}}g^{\frac{2m\left(p+2\right)}{2m-p}}(x)dx+\int_{B_R}\left|f(x)\right|^\frac{2m\left(p+2\right)}{2mp+p-2}dx+\left|B_R\right|+1\right],
\end{eqnarray}

which is the desired a priori estimate.
\\
\medskip

{\bf Step 2: the approximation.}\\
As we did in the second step of the proof of Theorem \ref{CGPThm1}, let us consider an open set $\Omega'\Subset\Omega$ and, for any $\varepsilon\in\left(0, d\left(\Omega', \partial\Omega\right)\right)$, a standard family of mollifiers $\Set{\phi_\varepsilon}_\varepsilon$.\\
Let us consider a ball $B_{\tilde{R}}=B_{\tilde{R}}\left(x_0\right)\Subset\Omega'$ with $\tilde{R}<1$ and, for each $\varepsilon$, the functional

\begin{equation}\label{modenergymeps}
	\mathcal{F}_{m, \varepsilon}\left(w, B_{\tilde{R}}\right)=\int_{B_{\tilde{R}}}\left[F_\varepsilon\left(x, Dw(x)\right)-f_{\varepsilon}(x)\cdot w(x)+\left(\left|w(x)\right|-a\right)^{2m}_+\right]dx,
\end{equation}

where $F_\varepsilon$ is defined as in \eqref{Fepsdef} and $f_\varepsilon$ is defined as in \eqref{fepsdef}.

With this choices, we have
\begin{equation}\label{convFm}
	\int_{B_{\tilde{R}}}F_\varepsilon\left(x, \xi\right)dx\to\int_{ B_{\tilde{R}}}F\left(x, \xi\right)dx,\qquad\mbox{ as }\varepsilon\to0
\end{equation}

for any $\xi\in\R^{n\times N}.$\\

Moreover, since $f\in L^{\frac{2m\left(p+2\right)}{2mp+p-2}}_{\loc}\left(\Omega\right)$, then

\begin{equation}\label{convfortefm1}
f_{\varepsilon}\to f\qquad\mbox{ strongly in }L^{\frac{2m\left(p+2\right)}{2mp+p-2}}\left(B_{\tilde{R}}\right),\mbox{ as }\varepsilon\to0.
\end{equation}
 
Let us observe that

$$
\frac{2m\left(p+2\right)}{2mp+p-2}\ge\frac{2m}{2m-1}
$$ 

if and only if 

$$
\left(2m-1\right)\left(p+2\right)\ge2mp+p-2, 
$$

i.e.

$$
2m\ge p, 
$$

which is true for any $m>1$, as long as $1<p<2$.

For this reason, $f\in L^\frac{2m}{2m-1}_{\loc}\left(B_{\tilde{R}}\right)$, and we also have

\begin{equation}\label{convfortefm2}
	f_{\varepsilon}\to f\qquad\mbox{ strongly in }L^{\frac{2m}{2m-1}}\left(B_{\tilde{R}}\right),\mbox{ as }\varepsilon\to0.
\end{equation} 
Again, as in the proof of Theorem \ref{CGPThm1}, thanks to \eqref{F1}--\eqref{F4}, for any $\varepsilon$, we have the validity of \eqref{F1eps}--\eqref{F4eps}, where $g_\varepsilon$ is defined in \eqref{gepsdef}.

In this case, since $g\in L^{\frac{2m\left(p+2\right)}{2mp-p}}_\loc\left(\Omega\right)$, we have

\begin{equation}\label{convgmeps}
	g_{\varepsilon}\to g\qquad\mbox{ strongly in }L^{\frac{2m\left(p+2\right)}{2mp-p}}\left(B_{\tilde{R}}\right)\mbox{ as }\varepsilon\to0.
\end{equation}

Let $v_\varepsilon\in \left(v+W^{1,p}_{0}\left(B_{\tilde{R}}\right)\right)\cap L^{2m}\left(B_R\right)$ be the solution to

$$
\min\Set{\mathcal{F}_{m, \varepsilon}\left(w,B_{\tilde{R}}\right): w\in  \left(v+W^{1,p}_{0}\left(B_{\tilde{R}}\right)\right)\cap L^{2m}\left(B_{\tilde{R}}\right)},
$$

where $v\in W^{1,p}_{\loc}\left(\Omega\right)\cap L^{2m}\left(B_R\right)$ is a local minimizer of \eqref{modenergym}.\\

By virtue of the minimality of $v_\varepsilon$, we have

\begin{eqnarray}\label{minimalitym*}
	&&\int_{B_{\tilde{R}}}\left[F_\varepsilon\left(x, Dv_\varepsilon(x)\right) +\left(\left|v_\varepsilon(x)\right|-a\right)^{2m}_+\right]dx\cr\cr
	&\le&\int_{B_{\tilde{R}}}\left[F_\varepsilon\left(x, Dv(x)\right)+f_{\varepsilon}(x)\cdot \left(v_\varepsilon(x)-v(x)\right)+\left(\left|v(x)\right|-a\right)^{2m}_+\right]dx\cr\cr
	&\le&\int_{B_{\tilde{R}}}\left[F_\varepsilon\left(x, Dv(x)\right)+\left|f_{\varepsilon}(x)\right|\cdot \left|v_\varepsilon(x)-v(x)\right|+\left(\left|v(x)\right|-a\right)^{2m}_+\right]dx.
\end{eqnarray}

Now, using H\"{o}lder's and Young's inequalities with exponents $\left(2m, \frac{2m}{2m-1}\right)$, we get

\begin{eqnarray}\label{controlloterminenotom}
	&&\int_{B_{\tilde{R}}}\left|f_{\varepsilon}(x)\right|\cdot \left|v_\varepsilon(x)-v(x)\right|dx\cr\cr
	&\le&\int_{B_{\tilde{R}}}\left|f_{\varepsilon}(x)\right|\left|v_\varepsilon(x)\right|dx+\int_{B_{\tilde{R}}}\left|f_{\varepsilon}(x)\right|\left|v(x)\right|dx\cr\cr
	&=&\int_{B_{\tilde{R}}}\left|f_{\varepsilon}(x)\right|\left(\left|v_\varepsilon(x)\right|-a\right)dx+\int_{B_{\tilde{R}}}a\left|f_{\varepsilon}(x)\right|dx+\int_{B_{\tilde{R}}}\left|f_{\varepsilon}(x)\right|\left|v(x)\right|dx\cr\cr
	&=&\int_{B_{\tilde{R}}\cap\Set{\left|v_\varepsilon\right|\ge a}}\left|f_{\varepsilon}(x)\right|\left(\left|v_\varepsilon(x)\right|-a\right)dx+\int_{B_{\tilde{R}}\cap\Set{\left|v_\varepsilon\right|<a}}\left|f_{\varepsilon}(x)\right|\left(\left|v_\varepsilon(x)\right|-a\right)dx\cr\cr
	&&+\int_{B_{\tilde{R}}}\left|f_{\varepsilon}(x)\right|\left(\left|v(x)\right|+a\right)dx\cr\cr
	&\le&\int_{B_{\tilde{R}}\cap\Set{\left|v_\varepsilon\right|\ge a}}\left|f_{\varepsilon}(x)\right|\left(\left|v_\varepsilon(x)\right|-a\right)_+dx+\int_{B_{\tilde{R}}}\left|f_{\varepsilon}(x)\right|\left(\left|v(x)\right|+a\right)dx\cr\cr
	&\le&\int_{B_{\tilde{R}}}\left|f_{\varepsilon}(x)\right|\left(\left|v_\varepsilon(x)\right|-a\right)_+dx+\int_{B_{\tilde{R}}}\left|f_{\varepsilon}(x)\right|\left(\left|v(x)\right|+a\right)dx\cr\cr
	&\le&\left(\int_{B_{\tilde{R}}}\left|f_{\varepsilon}(x)\right|^\frac{2m}{2m-1}dx\right)^\frac{2m-1}{2m}\cdot\left(\int_{B_{\tilde{R}}}\left(\left|v_\varepsilon(x)\right|-a\right)_+^{2m}dx\right)^\frac{1}{2m}\cr\cr
	&&+\left(\int_{B_{\tilde{R}}}\left|f_{\varepsilon}(x)\right|^\frac{2m}{2m-1}dx\right)^\frac{2m-1}{2m}\cdot\left(\int_{B_{\tilde{R}}}\left(\left|v(x)\right|+a\right)^{2m}dx\right)^\frac{1}{2m}\cr\cr
	&\le&c_\sigma\int_{B_{\tilde{R}}}\left|f_{\varepsilon}(x)\right|^\frac{2m}{2m-1}dx+\sigma\int_{B_{\tilde{R}}}\left(\left|v_\varepsilon(x)\right|-a\right)_+^{2m}dx\cr\cr
	&&+c\int_{B_{\tilde{R}}}\left(\left|v(x)\right|+a\right)^{2m}dx,
\end{eqnarray}

where $\sigma>0$ will be chosen later.\\
Plugging \eqref{controlloterminenotom} into \eqref{minimalitym*}, and choosing a sufficiently small $\sigma$, we get

\begin{eqnarray}\label{minimalitym**}
	&&\int_{B_{\tilde{R}}}\left[F_\varepsilon\left(x, Dv_\varepsilon(x)\right) +c\left(\left|v_\varepsilon(x)\right|-a\right)^{2m}_+\right]dx\cr\cr
	&\le&\int_{B_{\tilde{R}}}\left[F_\varepsilon\left(x, Dv(x)\right)+c\left(\left|v(x)\right|-a\right)^{2m}_+\right]dx\cr\cr
	&&+c\int_{B_{\tilde{R}}}\left|f_{\varepsilon}(x)\right|^\frac{2m}{2m-1}dx+c\int_{B_{\tilde{R}}}\left(\left|v(x)\right|+a\right)^{2m}dx. 
\end{eqnarray}

Using the right-hand side inequality in \eqref{F1eps} in \eqref{minimalitym**}, we have

\begin{eqnarray}\label{minimalitym}
	&&\int_{B_{\tilde{R}}}\left[F_\varepsilon\left(x, Dv_\varepsilon(x)\right) +c\left(\left|v_\varepsilon(x)\right|-a\right)^{2m}_+\right]dx\cr\cr
	&\le&L\int_{B_{\tilde{R}}}\left[\left(\mu^2+\left|Dv(x)\right|^2\right)^\frac{p}{2}+c\left(\left|v(x)\right|-a\right)^{2m}_+\right]dx\cr\cr
	&&+c\int_{B_{\tilde{R}}}\left|f_{\varepsilon}(x)\right|^\frac{2m}{2m-1}dx+c\int_{B_{\tilde{R}}}\left(\left|v(x)\right|+a\right)^{2m}dx.
\end{eqnarray}

Now, by the left-hand side inequality in \eqref{F1eps}, we get

\begin{eqnarray}\label{W1pBoundm}
	\ell\int_{B_{\tilde{R}}}\left(\mu^2+\left|Dv_\varepsilon(x)\right|^2\right)^\frac{p}{2}dx&\le&\int_{B_{\tilde{R}}}F_\varepsilon\left(x, Dv_\varepsilon(x)\right) dx\cr\cr
	&\le&\int_{B_{\tilde{R}}}\left[F_\varepsilon\left(x, Dv_\varepsilon(x)\right) +\left(\left|v_\varepsilon(x)\right|-a\right)^{2m}_+\right]dx\cr\cr
	&\le&L\int_{B_{\tilde{R}}}\left[\left(\mu^2+\left|Dv(x)\right|^2\right)^\frac{p}{2}+\left(\left|v(x)\right|-a\right)^{2m}_+\right]dx\cr\cr
	&&+c\int_{B_{\tilde{R}}}\left|f_{\varepsilon}(x)\right|^\frac{2m}{2m-1}dx\cr\cr
	&&+c\int_{B_{\tilde{R}}}\left(\left|v(x)\right|+a\right)^{2m}dx,
\end{eqnarray}

and this, by \eqref{convfortefm2}, means that $\Set{v_\varepsilon}_\varepsilon$ is bounded in $W^{1, p}\left(B_{\tilde{R}}\right)$, independently of $\varepsilon$, so there exists a function $\tilde{v}\in W^{1, p}\left(B_{\tilde{R}}\right)$ such that, up to a subsequence, we have

$$
v_\varepsilon\rightharpoonup \tilde{v}\qquad\mbox{ weakly in }W^{1, p}\left(B_{\tilde{R}}\right),
$$

$$
v_\varepsilon\to \tilde{v} \qquad\mbox{ strongly in }L^{p}\left(B_{\tilde{R}}\right), 
$$

and

$$
v_\varepsilon\to \tilde{v} \qquad\mbox{ almost everywhere in }B_{\tilde{R}}, 
$$
as $\varepsilon\to0$.\\
Moreover, we have

\begin{eqnarray}\label{L2mBound*}
	\int_{B_{\tilde{R}}}\left|v_\varepsilon(x)\right|^{2m}dx&\le&\int_{B_{\tilde{R}}\cap\Set{\left|v_\varepsilon\right|< a}}\left|v_\varepsilon(x)\right|^{2m}dx+\int_{B_{\tilde{R}}\cap\Set{\left|v_\varepsilon\right|\ge a}}\left|v_\varepsilon(x)\right|^{2m}dx\cr\cr
	&\le&\int_{B_{\tilde{R}}\cap\Set{\left|v_\varepsilon\right|< a}}\left|v_\varepsilon(x)\right|^{2m}dx+\int_{B_{\tilde{R}}\cap\Set{\left|v_\varepsilon\right|\ge a}}\left(\left|\left|v_\varepsilon(x)\right|-a\right|+a\right)^{2m}dx\cr\cr
	&\le&\int_{B_{\tilde{R}}\cap\Set{\left|v_\varepsilon\right|< a}}a^{2m}dx+c\int_{B_{\tilde{R}}\cap\Set{\left|v_\varepsilon\right|\ge a}}\left(\left|v_\varepsilon(x)\right|-a\right)^{2m}dx\cr\cr
	&&+c\int_{B_{\tilde{R}}\cap\Set{\left|v_\varepsilon\right|\ge a}}a^{2m}dx\cr\cr
	&\le&c\int_{B_{\tilde{R}}}a^{2m}dx+c\int_{B_{\tilde{R}}}\left(\left|v_\varepsilon(x)\right|-a\right)^{2m}_{+}dx,
\end{eqnarray}

and since \eqref{minimalitym} implies

\begin{eqnarray}\label{L2mBound}
	&&\int_{B_{\tilde{R}}}\left(\left|v_\varepsilon(x)\right|-a\right)^{2m}_+dx\cr\cr
	&\le&L\int_{B_{\tilde{R}}}\left[\left(\mu^2+\left|Dv(x)\right|^2\right)^\frac{p}{2}+\left(\left|v(x)\right|-a\right)^{2m}_+\right]dx\cr\cr
	&&+c\int_{B_{\tilde{R}}}\left|f_{\varepsilon}(x)\right|^\frac{2m}{2m-1}dx+c\int_{B_{\tilde{R}}}\left(\left|v(x)\right|+a\right)^{2m}dx,
\end{eqnarray}

by \eqref{convfortefm2}, and plugging \eqref{L2mBound} into \eqref{L2mBound*}, using dominate convergence theorem, we have

\begin{equation}\label{L2mconvforte}
	v_\varepsilon\to \tilde{v} \qquad\mbox{ strongly in }L^{2m}\left(B_{\tilde{R}}\right),\mbox{ as }\varepsilon\to0. 
\end{equation}

Since $v_\varepsilon$ is a local minimizer of the functional \eqref{modenergymeps} and $g_{\varepsilon}, f_{\varepsilon}\in C^\infty\left(B_{\tilde{R}}\right)$, we have 

$$
V_p\left(Dv_\varepsilon\right)\in W^{1,2}_{\loc}\left(B_{\tilde{R}}\right), 
$$

and we can apply estimate \eqref{apriorimPf}, thus getting 

\begin{eqnarray}\label{Step2estimatem}
\int_{B_{\frac{r}{2}}}\left|DV_p\left(Dv_\varepsilon(x)\right)\right|^2dx&\le&\frac{c}{r^{\frac{p+2}{p}}}\left[\int_{B_{r}} \left(\mu^2+\left|Dv_\varepsilon\left(x\right)\right|^2\right)^\frac{p}{2}dx\right.\cr\cr
&&\left.+\left(\int_{B_{4 r}}\left|v_\varepsilon(x)\right|^{2m}dx\right)^\frac{1}{m+1}\cdot\left(\int_{B_{4 r}}\left(\mu^2+\left|Dv_\varepsilon(x)\right|^2\right)^\frac{p}{2}dx\right)^\frac{m}{m+1}\right.\cr\cr
&&+\left.\int_{B_{r}}g_\varepsilon^{\frac{2m\left(p+2\right)}{2m-p}}(x)dx+\int_{B_r}\left|f_\varepsilon(x)\right|^\frac{2m\left(p+2\right)}{2mp+p-2}dx+\left|B_r\right|+1\right],
\end{eqnarray}

for any ball $B_{4r}\Subset B_{\tilde{R}}$.\\
Applying Lemma \ref{differentiabilitylemma}, by \eqref{differentiabilityestimate} and \eqref{Step2estimatem}, recalling \eqref{convfortefm1}, \eqref{convgmeps}, \eqref{W1pBoundm}, \eqref{L2mBound*} and \eqref{L2mBound}, and by a covering argument, we infer that $v_\varepsilon$ is bounded in $W^{2, p}\left(B_{2r}\right)$, which implies

\begin{equation}\label{convvepsW1p}
	v_\varepsilon\to\tilde{v}\qquad\mbox{ strongly in }W^{1,p}\left(B_{4r}\right)
\end{equation}

and 

$$
v_\varepsilon\to\tilde{v}\qquad\mbox{ almost everywhere in }B_{4r},
$$

up to a subsequence, as $\varepsilon\to0$.\\
By virtue of the continuity of the function $\xi\mapsto DV_p\left(\xi\right)$, we also have

 $$DV_p\left(Dv_\varepsilon\right)\to DV_p\left(D\tilde{v}\right)\qquad\mbox{ almost everywhere in }B_{4r},\mbox{ as }\varepsilon\to0.$$

For what we discussed above, and recalling \eqref{convfortefm1}, \eqref{convgmeps}, \eqref{L2mconvforte} and \eqref{convvepsW1p}, thanks to the dominate convergence theorem, we can pass to the limit in \eqref{Step2estimatem} as $\varepsilon\to0$, thus getting

\begin{eqnarray}\label{boundVpepsm}
&&\int_{B_{\frac{r}{2}}}\left|DV_p\left(D\tilde{v}(x)\right)\right|^2dx\cr\cr
&\le&\frac{c}{r^{\frac{p+2}{p}}}\left[\int_{B_{r}} \left(\mu^2+\left|D\tilde{v}\left(x\right)\right|^2\right)^\frac{p}{2}dx\right.\cr\cr
&&\left.+\left(\int_{B_{4 r}}\left|\tilde{v}(x)\right|^{2m}dx\right)^\frac{1}{m+1}\cdot\left(\int_{B_{4 r}}\left(\mu^2+\left|D\tilde{v}(x)\right|^2\right)^\frac{p}{2}dx\right)^\frac{m}{m+1}\right.\cr\cr
&&+\left.\int_{B_{r}}g^{\frac{2m\left(p+2\right)}{2m-p}}(x)dx+\int_{B_r}\left|f(x)\right|^\frac{2m\left(p+2\right)}{2mp+p-2}dx+\left|B_r\right|+1\right].
\end{eqnarray}

Our next aim is to prove that $\tilde{v}=v$ a.e. in $B_{\tilde{R}}$.\\
	First, let us observe that, using H\"older's inequality with exponents $\left(2m, \frac{2m}{2m-1}\right)$, we get
	\begin{eqnarray*}\label{finalconv2*}
		&&\left|\int_{B_{\tilde{R}}}\left[f_{\varepsilon}(x)\cdot
		v(x)-f(x)\cdot
		v(x)\right]dx\right|\cr\cr&\le&\int_{B_{\tilde{R}}}\left|f_{\varepsilon}(x)-f(x)\right|\cdot
		\left|v(x)\right|dx\cr\cr
		&\le&\left(\int_{B_{\tilde{R}}}\left|f_{\varepsilon}(x)-f(x)\right|^\frac{2m}{2m-1}dx\right)^\frac{2m-1}{2m}\cdot\left(\int_{B_{\tilde{R}}}\left|v(x)\right|^{2m}dx\right)^\frac{1}{2m}, 
	\end{eqnarray*}
that, recalling \eqref{convfortefm2}, implies
\begin{equation}\label{finalconv2}
	\lim_{\varepsilon\to 0}\int_{B_{\tilde{R}}}f_{\varepsilon}(x)\cdot v(x)dx=\int_{B_{\tilde{R}}}f(x)\cdot v(x)dx.
\end{equation}
	The minimality of $v$, Fatou's Lemma, the lower semicontinuity of $\mathcal{F}_{m, \varepsilon}$ and the minimality of $v_\varepsilon$ imply
	
\begin{eqnarray*}
	&&\int_{B_{\tilde{R}}}\left[F\left(x,Dv(x)\right)-f(x)\cdot
	v(x)+\left(\left|v(x)\right|-a\right)^{2m}_+\right]dx\cr\cr 
	&\le& \int_{B_{\tilde{R}}}\left[F\left(x,D\tilde{v}(x)\right)-f(x)\cdot
	\tilde{v}(x)+\left(\left|\tilde{v}(x)\right|-a\right)^{2m}_+\right]dx\cr\cr
	&\le& \liminf_{\varepsilon\to 0} \int_{B_{\tilde{R}}}\left[F_\varepsilon\left(x,D\tilde{v}(x)\right)-f_{\varepsilon}(x)\cdot
	\tilde{v}(x)+\left(\left|\tilde{v}(x)\right|-a\right)^{2m}_+\right]dx\cr\cr
	&\le& \liminf_{\varepsilon\to 0} \int_{B_{\tilde{R}}}\left[F_\varepsilon\left(x,Dv_\varepsilon(x)\right)-f_{\varepsilon}(x)\cdot
	v_\varepsilon(x)+\left(\left|v_\varepsilon(x)\right|-a\right)^{2m}_+\right]dx\cr\cr
	&\le& \liminf_{\varepsilon\to 0} \int_{B_{\tilde{R}}}\left[F_\varepsilon\left(x,Dv(x)\right)-f_{\varepsilon}(x)\cdot
	v(x)+\left(\left|v(x)\right|-a\right)^{2m}_+\right]dx\cr\cr
	&=&\int_{B_{\tilde{R}}}\left[F\left(x,Dv(x)\right)-f(x)\cdot
	v(x)+\left(\left|v(x)\right|-a\right)^{2m}_+\right]dx,
\end{eqnarray*}
where the last inequality is a consequence of \eqref{convFm} and \eqref{finalconv2}.\\ 
Therefore $\mathcal{F}_m\left(D\tilde{v},B_{\tilde{R}}\right)=\mathcal{F}_m\left(Dv,B_{\tilde{R}}\right)$ and the strict convexity of the functional yields that $\tilde{v}=v$ a.e. in $B_{\tilde{R}}$.
So \eqref{boundVpepsm} and a covering argument yield \eqref{approxmestimate}.
\end{proof}

We conclude this section with some consequences of Theorem \ref{approxmthm}, which follow by Lemma \ref{differentiabilitylemma} and Remark \ref{rmk3}.

\begin{corollary}\label{corollarym}
	Let $\Omega\subset\R^n$ be a bounded open set, $m>1$, $a>0$ and $1<p<2$.\\
	Let $v\in W^{1, p}_{\loc}\left(\Omega, \R^N\right)\cap L^{2m}_{\loc}\left(\Omega, \R^N\right)$ be a local minimizer of the functional \eqref{modenergym}, under the assumptions \eqref{F1}--\eqref{F4}, with  
	
	$$f\in L^{\frac{2m\left(p+2\right)}{2mp+p-2}}_{\loc}\left(\Omega\right)\qquad\mbox{ and }\qquad g\in L^{\frac{2m\left(p+2\right)}{2m-p}}_{\loc}\left(\Omega\right).$$
	
	Then $v\in W^{2, p}_{\loc}\left(\Omega\right)$ and $Dv\in L^{\frac{m\left(p+2\right)}{m+1}}_{\loc}\left(\Omega\right)$.
\end{corollary}

\section{The case of bounded minimizers: proof of Theorem \ref{inftythm}}\label{Thm2pf}

The aim of this section is to prove Theorem \ref{inftythm}.\\
As we will see below, the proof is achieved by using an approximation argument which is based on the possibility to apply Theorem \ref{approxmthm} and pass to the limit as $m\to\infty$.

\begin{proof}[Proof of Theorem \ref{inftythm}]
	Arguing as in the second step of the proof of Theorem \ref{approxmthm}, let us consider an open set $\Omega'\Subset\Omega$ and, for any $\varepsilon\in\left(0, d\left(\Omega', \partial\Omega\right)\right)$, a standard family of mollifiers $\Set{\phi_\varepsilon}_\varepsilon$.\\
	Let $u\in W^{1, p}_\loc\left(\Omega\right)\cap L^\infty_\loc\left(\Omega\right)$ be a local minimizer of the functional \eqref{modenergy}, and let us consider a ball $B_{\tilde{R}}=B_{\tilde{R}}\left(x_0\right)\Subset\Omega'$, with $\tilde{R}<1$.\\ 
	For each $\varepsilon$ and any $m>1$, let us consider the functional $\mathcal{F}_{m, \varepsilon}$, defined by \eqref{modenergymeps}, where $F_\varepsilon$ and $f_\varepsilon$ are defined by \eqref{Fepsdef} and \eqref{fepsdef} respectively, and we fix
	
	\begin{equation}\label{afix}
		a=\left\Arrowvert u\right\Arrowvert_{L^\infty\left(B_{\tilde{R}}\right)}.
	\end{equation}
With these choices, we have \eqref{convFm} again, and since $f\in L^{\frac{p+2}{p}}_{\loc}\left(\Omega\right)$, we have

\begin{equation}\label{convfinf1}
	f_\varepsilon\to f \qquad\mbox{ strongly in }L^{\frac{p+2}{p}}\left(B_{\tilde{R}}\right),\qquad\mbox{ as }\varepsilon\to0.
\end{equation}

Again, thanks to \eqref{F1}--\eqref{F4}, for any $\varepsilon$, we have \eqref{F1eps}--\eqref{F4eps}, where $g_\varepsilon$ is defined like in \eqref{gepsdef}.

In this case, since $g\in L^{p+2}_\loc\left(\Omega\right)$, we have

\begin{equation}\label{convgeps}
	g_\varepsilon\to g\qquad\mbox{ strongly in }L^{p+2}\left(B_{\tilde{R}}\right),\mbox{ as }\varepsilon\to0.
\end{equation}

Let us observe that $f_\varepsilon\in L^{\frac{2m\left(p+2\right)}{2mp+p-2}}\left(B_{\tilde{R}}\right)$ for any $m>1$, and since
$$
\frac{2m\left(p+2\right)}{2mp+p-2}\searrow\frac{p+2}{p},\qquad\mbox{ as }m\to\infty,
$$
we have
\begin{equation}\label{convfinfepsm}
	\lim_{m\to\infty}\left(\int_{B_{\tilde{R}}}\left|f_\varepsilon(x)\right|^\frac{2m\left(p+2\right)}{2mp+p-2}dx\right)^\frac{2mp+p-2}{2m\left(p+2\right)}=\left(\int_{B_{\tilde{R}}}\left|f_\varepsilon(x)\right|^\frac{p+2}{p}dx\right)^\frac{p}{p+2},
\end{equation}
for any $\varepsilon$.

Similarly, then $g_\varepsilon\in L^{\frac{2m\left(p+2\right)}{2m-p}}\left(B_{\tilde{R}}\right)$ for any $m>1$ and for any $\varepsilon$, and we have 

\begin{equation}\label{gepsboundm}
	\lim_{m\to\infty}\left(\int_{B_{\tilde{R}}}\left|g_{\varepsilon}(x)\right|^{\frac{2m\left(p+2\right)}{2m-p}}dx\right)^{\frac{2m-p}{2m\left(p+2\right)}}=\left(\int_{B_{\tilde{R}}}\left|g_{\varepsilon}(x)\right|^{p+2}dx\right)^{\frac{1}{p+2}},
\end{equation}

for each $\varepsilon$.\\

Now, for each $\varepsilon$, and for each $m>1$, let $u_{m, \varepsilon}\in \left(u+W^{1, p}_0\left(B_{\tilde{R}}\right)\right)\cap L^{2m}\left(B_{\tilde{R}}\right)$ be the solution to 

$$
\min\Set{\mathcal{F}_{m, \varepsilon}\left(w,B_{\tilde{R}}\right): w\in  \left(u+W^{1,p}_{0}\left(B_{\tilde{R}}\right)\right)\cap L^{2m}\left(B_{\tilde{R}}\right)}.
$$

By virtue of the minimality of $u_{m, \varepsilon}$, recalling \eqref{afix}, we have 

\begin{eqnarray}\label{minimalityinf*}
	&&\int_{B_{\tilde{R}}}\left[F_\varepsilon\left(x, Du_{m, \varepsilon}(x)\right) +\left(\left|u_{m, \varepsilon}(x)\right|-a\right)^{2m}_+\right]dx\cr\cr
	&\le&\int_{B_{\tilde{R}}}\left[F_\varepsilon\left(x, Du(x)\right)+f_{\varepsilon}(x)\cdot \left(u_{m, \varepsilon}(x)-u(x)\right)+\left(\left|u(x)\right|-a\right)^{2m}_+\right]dx\cr\cr
	&\le&\int_{B_{\tilde{R}}}\left[F_\varepsilon\left(x, Du(x)\right)+\left|f_{\varepsilon}(x)\right|\cdot \left|u_{m, \varepsilon}(x)-u(x)\right|\right]dx.
\end{eqnarray}

Arguing as we did in \eqref{controlloterminenotom} and exploiting \eqref{afix}, we get 

\begin{eqnarray}\label{controlloterminenotoinf*}
	&&\int_{B_{\tilde{R}}}\left|f_{\varepsilon}(x)\right|\cdot \left|u_{m, \varepsilon}(x)-u(x)\right|dx\cr\cr
	&\le&\int_{B_{\tilde{R}}}\left|f_{\varepsilon}(x)\right|\left(\left|u_{m, \varepsilon}(x)\right|-a\right)_+dx+2a\int_{B_{\tilde{R}}}\left|f_{\varepsilon}(x)\right|dx\cr\cr
	&\le&\left(\int_{B_{\tilde{R}}}\left|f_{\varepsilon}(x)\right|^{\frac{2m\left(p+2\right)}{p\left(2m-1\right)}}dx\right)^{\frac{p\left(2m-1\right)}{2m\left(p+2\right)}}\cdot\left(\int_{B_{\tilde{R}}}\left(\left|u_{m, \varepsilon}(x)\right|-a\right)_+^{\frac{2m\left(p+2\right)}{4m+p}}dx\right)^\frac{4m+p}{2m\left(p+2\right)}\cr\cr
	&&+2a\int_{B_{\tilde{R}}}\left|f_{\varepsilon}(x)\right|dx,
\end{eqnarray}

where, in the last line, we used H\"{o}lder's Inequality with exponents $\left(\frac{2m\left(p+2\right)}{p\left(2m-1\right)}, \frac{2m\left(p+2\right)}{4m+p}\right)$.
Let us notice that all the integrals in the last line of \eqref{controlloterminenotoinf*} are finite, since $f_\varepsilon\in C^{\infty}\left(B_{\tilde{R}}\right)$ and $\frac{2m\left(p+2\right)}{4m+p}<2m$ for any $m>1$ as long as $1<p<2$.
So, since $u_{m, \varepsilon}\in L^{2m}\left(B_{\tilde{R}}\right)\hookrightarrow L^{\frac{2m\left(p+2\right)}{4m+p}}\left(B_{\tilde{R}}\right)$, using Young's Inequality with exponents $\left(2m, \frac{2m}{2m-1}\right)$, we have 

\begin{eqnarray}\label{controlloterminenotoinf}
	&&\int_{B_{\tilde{R}}}\left|f_{\varepsilon}(x)\right|\cdot \left|u_{m, \varepsilon}(x)-u(x)\right|dx\cr\cr
	&\le&c\left(\int_{B_{\tilde{R}}}\left|f_{\varepsilon}(x)\right|^{\frac{2m\left(p+2\right)}{p\left(2m-1\right)}}dx\right)^{\frac{p\left(2m-1\right)}{2m\left(p+2\right)}}\cdot\left(\int_{B_{\tilde{R}}}\left(\left|u_{m, \varepsilon}(x)\right|-a\right)_+^{2m}dx\right)^\frac{1}{2m}\cr\cr
	&&+2a\int_{B_{\tilde{R}}}\left|f_{\varepsilon}(x)\right|dx\cr\cr
	&\le&c_\sigma\left(\int_{B_{\tilde{R}}}\left|f_{\varepsilon}(x)\right|^{\frac{2m\left(p+2\right)}{p\left(2m-1\right)}}dx\right)^{\frac{p}{p+2}}+\sigma\int_{B_{\tilde{R}}}\left(\left|u_{m, \varepsilon}(x)\right|-a\right)_+^{2m}dx\cr\cr
	&&+2a\int_{B_{\tilde{R}}}\left|f_{\varepsilon}(x)\right|dx, 
\end{eqnarray}

for any $\sigma>0$.

Plugging \eqref{controlloterminenotoinf} into \eqref{minimalityinf*}, choosing a sufficiently small value of $\sigma$ and recalling \eqref{F1eps}, we get

\begin{eqnarray}\label{minimalityinf}
	&&\int_{B_{\tilde{R}}}\left[\left(\mu^2+\left|Du_{m, \varepsilon}(x)\right|^2\right)^\frac{p}{2}+\left(\left|u_{m, \varepsilon}(x)\right|-a\right)^{2m}_+\right]dx\cr\cr
	&\le&L\int_{B_{\tilde{R}}}\left(\mu^2+\left|Du(x)\right|^2\right)^\frac{p}{2}dx+c\left(\int_{B_{\tilde{R}}}\left|f_{\varepsilon}(x)\right|^{\frac{2m\left(p+2\right)}{p\left(2m-1\right)}}dx\right)^{\frac{p}{p+2}}.
\end{eqnarray}

Now let us notice that, since

$$
\frac{2m\left(p+2\right)}{p\left(2m-1\right)}\ge\frac{p+2}{p}
$$

for any $m>1$, and

$$
\frac{2m\left(p+2\right)}{p\left(2m-1\right)}\searrow\frac{p+2}{p},\qquad\mbox{ as }m\to\infty,
$$

we have

\begin{eqnarray*}
	\lim_{m\to\infty}\left(\int_{B_{\tilde{R}}}\left|f_{\varepsilon}(x)\right|^{\frac{2m\left(p+2\right)}{p\left(2m-1\right)}}dx\right)^{\frac{p\left(2m-1\right)}{2m\left(p+2\right)}}=\left(\int_{B_{\tilde{R}}}\left|f_{\varepsilon}(x)\right|^{\frac{p+2}{p}}dx\right)^{\frac{p}{p+2}}, 
\end{eqnarray*}

and so

\begin{eqnarray}\label{fmepsconvm}
	\lim_{m\to\infty}\left(\int_{B_{\tilde{R}}}\left|f_{\varepsilon}(x)\right|^{\frac{2m\left(p+2\right)}{p\left(2m-1\right)}}dx\right)^{\frac{p}{p+2}}&=&\lim_{m\to\infty}\left(\int_{B_{\tilde{R}}}\left|f_{\varepsilon}(x)\right|^{\frac{2m\left(p+2\right)}{p\left(2m-1\right)}}dx\right)^{\frac{p\left(2m-1\right)}{2m\left(p+2\right)}\cdot\frac{2m}{\left(2m-1\right)}}\cr\cr
	&=&\left(\int_{B_{\tilde{R}}}\left|f_{\varepsilon}(x)\right|^{\frac{p+2}{p}}dx\right)^{\frac{p}{p+2}}, 
\end{eqnarray}

for any $\varepsilon$.\\
Hence, for any $\varepsilon$, the second integral in the right-hand side of \eqref{minimalityinf} can be bounded independently of $m$.\\
This implies that, for each $\varepsilon$, $\Set{u_{m, \varepsilon}}_m$ is bounded in $W^{1, p}\left(B_{\tilde{R}}\right)$, and so there exists a family of functions $\Set{u_{\varepsilon}}_\varepsilon\subset W^{1, p}\left(B_{\tilde{R}}\right)$ such that

\begin{equation*}\label{convminfw}
	u_{m, \varepsilon}\rightharpoonup u_\varepsilon\qquad\mbox{ weakly in }W^{1, p}\left(B_{\tilde{R}}\right),
\end{equation*} 

and so 

\begin{equation*}\label{convminfs}
	u_{m, \varepsilon}\to u_\varepsilon\qquad\mbox{ strongly in }L^{p}\left(B_{\tilde{R}}\right),
\end{equation*} 

and 

\begin{equation}\label{convminfae}
	u_{m, \varepsilon}\to u_\varepsilon\qquad\mbox{ almost everywhere in }B_{\tilde{R}},
\end{equation} 

as $m\to\infty$, up to a subsequence.\\
In particular, by \eqref{minimalityinf}, \eqref{convfinf1} and \eqref{fmepsconvm}, the set of functions $\Set{u_{\varepsilon}}_\varepsilon$ is bounded in $W^{1, p}\left(B_{\tilde{R}}\right)$, and so there exists a function $v\in W^{1, p}\left(B_{\tilde{R}}\right)$ such that

\begin{equation*}\label{conveps0inf}
	u_{\varepsilon}\rightharpoonup v\qquad\mbox{ weakly in }W^{1, p}\left(B_{\tilde{R}}\right),\mbox{ as }\varepsilon\to0.
\end{equation*} 

So we have

\begin{equation*}\label{conveps0infs}
	u_{\varepsilon}\to v\qquad\mbox{ strongly in }L^{p}\left(B_{\tilde{R}}\right)
\end{equation*} 

and 

\begin{equation}\label{conveps0infae}
	u_{\varepsilon}\to v \qquad\mbox{ almost everywhere in }B_{\tilde{R}},
\end{equation} 

up to a subsequence, as $\varepsilon\to0$.\\

On the other hand, \eqref{minimalityinf} implies

\begin{eqnarray}\label{L2mboundinf*}
	&&\int_{B_{\tilde{R}}}\left(\left|u_{m, \varepsilon}(x)\right|-a\right)^{2m}_+dx\cr\cr
	&\le&L\int_{B_{\tilde{R}}}\left(\mu^2+\left|Du(x)\right|^2\right)^\frac{p}{2}dx+c\left(\int_{B_{\tilde{R}}}\left|f_{\varepsilon}(x)\right|^{\frac{2m\left(p+2\right)}{p\left(2m-1\right)}}dx\right)^{\frac{p}{p+2}},
\end{eqnarray}

and this bound is independent of $m$ by virtue of \eqref{fmepsconvm}.\\

One can easily check that, for any $m>1$, we have

\begin{eqnarray}\label{L2mboundinf**}
	\int_{B_{\tilde{R}}}\left|u_{m, \varepsilon}(x)\right|^{2m}dx\le\int_{B_{\tilde{R}}}\left(\left|u_{m, \varepsilon}(x)\right|-a\right)_+^{2m}dx+c\int_{B_{\tilde{R}}}a^{2m}dx,
\end{eqnarray}

and so, by virtue of \eqref{L2mboundinf*}, for any $m>1$, we get

\begin{eqnarray*}
\int_{B_{\tilde{R}}}\left|u_{m, \varepsilon}(x)\right|^{2m}dx&\le&L\int_{B_{\tilde{R}}}\left(\mu^2+\left|Du(x)\right|^2\right)^\frac{p}{2}dx\cr\cr
&&+c\left(\int_{B_{\tilde{R}}}\left|f_{\varepsilon}(x)\right|^{\frac{2m\left(p+2\right)}{p\left(2m-1\right)}}dx\right)^{\frac{p}{p+2}}+c\int_{B_{\tilde{R}}}a^{2m}dx
\end{eqnarray*}

and

\begin{eqnarray}\label{L2mboundinf***}
	&&\left(\int_{B_{\tilde{R}}}\left|u_{m, \varepsilon}(x)\right|^{2m}dx\right)^\frac{1}{2m}\cr\cr
	&\le&c\left[\int_{B_{\tilde{R}}}\left(\mu^2+\left|Du(x)\right|^2\right)^\frac{p}{2}dx\right]^\frac{1}{2m}+c\left[\int_{B_{\tilde{R}}}a^{2m}dx\right]^\frac{1}{2m}\cr\cr
	&&+c\left(\int_{B_{\tilde{R}}}\left|f_{\varepsilon}(x)\right|^{\frac{2m\left(p+2\right)}{p\left(2m-1\right)}}dx\right)^{\frac{p\left(2m-1\right)}{2m\left(p+2\right)}\cdot\frac{1}{2m-1}}.
\end{eqnarray}

Now, if we pass to the $\limsup$ as $m\to\infty$ at both sides of \eqref{L2mboundinf***}, recalling \eqref{afix} and \eqref{fmepsconvm}, we get

\begin{equation*}
	\limsup_{m\to\infty}\left(\int_{B_{\tilde{R}}}\left|u_{m, \varepsilon}(x)\right|^{2m}dx\right)^\frac{1}{2m}\le c\left\Arrowvert u\right\Arrowvert_{L^\infty\left(B_{\tilde{R}}\right)},
\end{equation*}

and similarly, for any ball $B_{4r}\Subset B_{\tilde{R}}$, we have

\begin{equation*}
	\limsup_{m\to\infty}\left(\int_{B_{4r}}\left|u_{m, \varepsilon}(x)\right|^{2m}dx\right)^\frac{1}{2m}\le c\left\Arrowvert u\right\Arrowvert_{L^\infty\left(B_{4r}\right)}, 
\end{equation*}

which implies

\begin{equation}\label{limsupmeps}
	\limsup_{m\to\infty}\left(\int_{B_{4r}}\left|u_{m, \varepsilon}(x)\right|^{2m}dx\right)^\frac{1}{m+1}\le c\left\Arrowvert u\right\Arrowvert_{L^\infty\left(B_{4r}\right)}^2.
\end{equation}

Since, for any $m>1$ and for any $\varepsilon$, $u_{m, \varepsilon}\in \left(u+W^{1, p}_0\left(B_{\tilde{R}}\right)\right)\cap L^{2m}\left(B_{\tilde{R}}\right)$ is a minimizer of a functional of the form \eqref{modenergym}, which satisfies \eqref{F1eps}--\eqref{F4eps}, $g_\varepsilon\in L^{\frac{2m\left(p+2\right)}{2m-p}}\left(B_{\tilde{R}}\right)$ and $f_\varepsilon\in L^{\frac{2m\left(p+2\right)}{2mp+p-2}}\left(B_{\tilde{R}}\right)$, we can apply Theorem \ref{approxmthm}, and by \eqref{approxmestimate}, we get

\begin{eqnarray}\label{DVpmeps}
	&&\int_{B_{\frac{r}{2}}}\left|DV_p\left(Du_{m, \varepsilon}(x)\right)\right|^2dx\cr\cr
	&\le&\frac{c}{r^{\frac{p+2}{p}}}\left[\int_{B_{r}} \left(\mu^2+\left|Du_{m, \varepsilon}\left(x\right)\right|^2\right)^\frac{p}{2}dx\right.\cr\cr
	&&\left.+\left(\int_{B_{4 r}}\left|u_{m, \varepsilon}(x)\right|^{2m}dx\right)^\frac{1}{m+1}\cdot\left(\int_{B_{4 r}}\left(\mu^2+\left|Du_{m, \varepsilon}(x)\right|^2\right)^\frac{p}{2}dx\right)^\frac{m}{m+1}\right.\cr\cr
	&&+\left.\int_{B_{r}}g_\varepsilon^{\frac{2m\left(p+2\right)}{2m-p}}(x)dx+\int_{B_r}\left|f_\varepsilon(x)\right|^\frac{2m\left(p+2\right)}{2mp+p-2}dx+\left|B_r\right|+1\right],
\end{eqnarray}

for any ball $B_{4r}\Subset B_{\tilde{R}}$.\\
Moreover, we can use Lemma \ref{differentiabilitylemma} and \eqref{differentiabilityestimate}, thus getting

\begin{eqnarray}\label{D2Lpmestimatemeps}
	&&\int_{B_{\frac{r}{2}}}\left|D^2u_{m, \varepsilon}(x)\right|^pdx\cr\cr
&\le&\frac{c}{r^{\frac{p+2}{p}}}\left[\int_{B_{r}} \left(\mu^2+\left|Du_{m, \varepsilon}\left(x\right)\right|^2\right)^\frac{p}{2}dx\right.\cr\cr
&&\left.+\left(\int_{B_{4 r}}\left|u_{m, \varepsilon}(x)\right|^{2m}dx\right)^\frac{1}{m+1}\cdot\left(\int_{B_{4 r}}\left(\mu^2+\left|Du_{m, \varepsilon}(x)\right|^2\right)^\frac{p}{2}dx\right)^\frac{m}{m+1}\right.\cr\cr
&&+\left.\int_{B_{r}}g_\varepsilon^{\frac{2m\left(p+2\right)}{2m-p}}(x)dx+\int_{B_r}\left|f_\varepsilon(x)\right|^\frac{2m\left(p+2\right)}{2mp+p-2}dx+\left|B_r\right|+1\right].
\end{eqnarray}

By virtue of \eqref{convfinfepsm}, \eqref{gepsboundm}, \eqref{minimalityinf}, \eqref{fmepsconvm} and \eqref{limsupmeps}, all the integrals in the right-hand side of \eqref{D2Lpmestimatemeps} are bounded independently of $m$: for this reason, for each $\varepsilon$, the set of functions $\Set{u_{m, \varepsilon}}_m$ is bounded in $W^{2,p}\left(B_{\frac{r}{2}}\right)$, and since the ball $B_{4r}$ is arbitrary, a covering argument implies

\begin{equation}\label{W1pstrongmueps}
	u_{m, \varepsilon}\to u_\varepsilon\qquad\mbox{ strongly in }W^{1, p}\left(B_{4r}\right),
\end{equation} 

which gives

\begin{equation*}
	Du_{m, \varepsilon}\to Du_\varepsilon\qquad\mbox{ almost everywhere in }B_{4r},
\end{equation*} 

up to a subsequence, as $m\to\infty$.\\
So, passing to the limit as $m\to\infty$, recalling \eqref{minimalityinf} and \eqref{fmepsconvm}, we also get
	\begin{eqnarray}\label{mapproxinfW1pLim}
		&&\int_{B_{2r}}\left(\mu^2+\left|Du_{\varepsilon}(x)\right|^2\right)^\frac{p}{2}dx\cr\cr
		&\le&L\int_{B_{\tilde{R}}}\left(\mu^2+\left|Du(x)\right|^2\right)^\frac{p}{2}dx+c\left(\int_{B_{\tilde{R}}}\left|f_{\varepsilon}(x)\right|^{\frac{p+2}{p}}dx\right)^{\frac{p}{p+2}}.
	\end{eqnarray}

Therefore, since, by virtue of the continuity of $\xi\mapsto DV_p(\xi)$, we also have

$$
DV_p\left(Du_{m, \varepsilon}\right)\to DV_p\left(Du_{\varepsilon}\right)\qquad\mbox{ almost everywhere in }B_{2r},\mbox{ as }m\to\infty, 
$$

and we can apply Fatou's Lemma passing to the $\limsup$ as $m\to\infty$ in \eqref{DVpmeps}, using \eqref{convfinfepsm}, \eqref{gepsboundm}, \eqref{limsupmeps} and \eqref{W1pstrongmueps}, we get

\begin{eqnarray}\label{DVpeps}
&&\int_{B_{\frac{r}{2}}}\left|DV_p\left(Du_{\varepsilon}(x)\right)\right|^2dx\cr\cr
&\le&\frac{c\left\Arrowvert u\right\Arrowvert_{L^\infty\left(B_{4r}\right)}}{r^{\frac{p+2}{p}}}\left[\int_{B_{4r}} \left(\mu^2+\left|Du_{\varepsilon}\left(x\right)\right|^2\right)^\frac{p}{2}dx\right.\cr\cr
&&+\left.\int_{B_{r}}g_\varepsilon^{p+2}(x)dx+\int_{B_r}\left|f_\varepsilon(x)\right|^\frac{p+2}{p}dx+\left|B_r\right|+1\right],
\end{eqnarray}
where we also used the fact that $r<\tilde{R}<1$.\\
Now, since, by virtue of \eqref{convfinf1}, \eqref{convgeps}, and \eqref{mapproxinfW1pLim}, all the integrals in the right-hand side of \eqref{DVpeps} can be bounded independently of $\varepsilon$, arguing like in the proof of Lemma \ref{differentiabilitylemma}, it is possible to prove that $\Set{u_\varepsilon}_\varepsilon$ is bounded in $W^{2,p}\left(B_{\frac{r}{2}}\right)$, and since $r$ is arbitrary, a covering argument implies

\begin{equation*}
	u_{\varepsilon}\to v\qquad\mbox{ strongly in }W^{1, p}\left(B_{4r}\right),
\end{equation*} 

and

\begin{equation*}
	Du_{\varepsilon}\to Dv\qquad\mbox{ almost everywhere in }B_{4r},
\end{equation*} 

as $\varepsilon\to0$.\\
Since, by virtue of the continuity of $\xi\mapsto DV_p(\xi)$, we also have

$$
DV_p\left(Du_{\varepsilon}\right)\to DV_p\left(Dv\right)\qquad\mbox{ almost everywhere in }B_{4r},\mbox{ as }\varepsilon\to0, 
$$

using Fatou's Lemma in \eqref{DVpeps}, we get

\begin{eqnarray}\label{DVpinfty}
&&\int_{B_{\frac{r}{2}}}\left|DV_p\left(Dv(x)\right)\right|^2dx\cr\cr
&\le&\frac{c\left\Arrowvert u\right\Arrowvert_{L^\infty\left(B_{4r}\right)}}{r^{\frac{p+2}{p}}}\left[\int_{B_{4r}} \left(\mu^2+\left|Dv\left(x\right)\right|^2\right)^\frac{p}{2}dx\right.\cr\cr
&&+\left.\int_{B_{r}}g^{p+2}(x)dx+\int_{B_r}\left|f(x)\right|^\frac{p+2}{p}dx+\left|B_r\right|+1\right],
\end{eqnarray}

The last step to get the conclusion consists in proving that $u=v$ a.e. on $B_{\tilde{R}}$.\\
Since we have

\begin{equation*}
\int_{B_{\tilde{R}}}\left|f_\varepsilon(x)\cdot u(x)-f(x)\cdot u(x)\right|dx\le\left\Arrowvert u\right\Arrowvert_{L^\infty\left(B_{\tilde{R}}\right)}\int_{B_{\tilde{R}}}\left|f_\varepsilon(x)-f(x)\right|dx, 
\end{equation*}

by virtue of \eqref{convfinf1}, we get

\begin{equation}\label{lasteps0*}
	\lim_{\varepsilon\to 0}\int_{B_{\tilde{R}}}f_\varepsilon(x)\cdot u(x)dx=\int_{B_{\tilde{R}}}f(x)\cdot u(x)dx.
\end{equation}

Using the minimality of $u$, the lower semicontinuity of the functional $\mathcal{F}$, the minimality of $u_{m, \varepsilon}$ for $\mathcal{F}_{m,\varepsilon}$ and the lower semicontinuity of this functional and recalling \eqref{afix}, we get
\begin{eqnarray}\label{semicont}
	&&\int_{B_{\tilde{R}}}\left[F\left(x, Du(x)\right)-f(x)\cdot u(x)\right] dx\cr\cr
	&\le&\int_{B_{\tilde{R}}}\left[F\left(x, Dv(x)\right)-f(x)\cdot v(x)\right] dx\cr\cr
	&\le&\liminf_{\varepsilon\to 0}\int_{B_{\tilde{R}}}\left[F_\varepsilon\left(x, Du_\varepsilon(x)\right)-f_\varepsilon(x)\cdot u_\varepsilon(x)\right] dx\cr\cr
	&\le&\liminf_{\varepsilon\to 0}\liminf_{m\to\infty}\int_{B_{\tilde{R}}}\left[F_\varepsilon\left(x, Du_{m, \varepsilon}(x)\right)-f_\varepsilon(x)\cdot u_{m, \varepsilon}(x)\right]dx\cr\cr
	&\le&\liminf_{\varepsilon\to 0}\liminf_{m\to\infty}\int_{B_{\tilde{R}}}\left[F_\varepsilon\left(x, Du_{m, \varepsilon}(x)\right)-f_\varepsilon(x)\cdot u_{m, \varepsilon}(x)+\left(\left|u_{m, \varepsilon}(x)\right|-a\right)_+^{2m}\right]dx\cr\cr
	&\le&\liminf_{\varepsilon\to 0}\int_{B_{\tilde{R}}}\left[F_\varepsilon\left(x, Du(x)\right)-f_\varepsilon(x)\cdot u(x)\right]dx\cr\cr
	&=&\int_{B_{\tilde{R}}}\left[F\left(x, Du(x)\right)-f(x)\cdot u(x)\right]dx,
\end{eqnarray}

where the last equality is a consequence of \eqref{convFm} and \eqref{lasteps0*}.
Therefore, all the inequalities in \eqref{semicont} hold as equalities, and we get

$$
\mathcal{F}\left(u, B_{\tilde{R}}\right)=\mathcal{F}\left(v, B_{\tilde{R}}\right).
$$

So, by virtue of the strict convexity of $F$ with respect to the gradient variable, this implies $u=v$ a.e. on $B_{\tilde{R}}$.
By virtue of \eqref{DVpinfty} and a standard covering argument, we get \eqref{inftyestimate}.
\end{proof}

We conclude this section with some consequences of Theorem \ref{inftythm}, that can be proved applying Lemma \ref{differentiabilitylemma} and estimate \eqref{2.2GP} rispectively. 

\begin{corollary}\label{corollaryinf}
	Let $\Omega\subset\R^n$ be a bounded open set and $1<p<2$.\\
	Let $u\in W^{1, p}_{\loc}\left(\Omega, \R^N\right)\cap L^{\infty}_{\loc}\left(\Omega, \R^N\right)$ be a local minimizer of the functional \eqref{modenergy}, under the assumptions \eqref{F1}--\eqref{F4}, with  
	
	$$f\in L^{\frac{p+2}{p}}_{\loc}\left(\Omega\right)\qquad\mbox{ and }\qquad g\in L^{p+2}_{\loc}\left(\Omega\right).$$
	
	Then $u\in W^{2, p}_{\loc}\left(\Omega\right)$ and $Du\in L^{p+2}_{\loc}\left(\Omega\right)$.
\end{corollary}
\printbibliography
\end{document}